\documentclass[12pt]{article}
\usepackage{a4wide}
\usepackage{graphicx}
\usepackage{amsmath}
\usepackage{amssymb}
\usepackage{amsthm}
\usepackage{enumerate}

\newtheorem{theorem}{Theorem}[section]
\newtheorem{definition}[theorem]{Definition}

\newtheorem{lemma}[theorem]{Lemma}
\newtheorem{claim}[theorem]{Claim}
\newtheorem{corollary}[theorem]{Corollary}
\newtheorem{conjecture}[theorem]{Conjecture}

\newtheorem{problem}[theorem]{Problem}

\newcommand{\comment}[1]{}
\newcommand{\hyper}{\mathcal}
\newcommand{\N}{\mathbb{N}}

\begin{document}

\title{Edge disjoint Hamiltonian cycles in highly connected tournaments}

\author{\large{Alexey Pokrovskiy} 
\\
\\ Methods for Discrete Structures,
\\ {Freie Universit\"at,} 
\\ Berlin, Germany.
\\ {Email: \texttt{alja123@gmail.com}}
}

\maketitle

\begin{abstract}
Thomassen conjectured that there is a function $f(k)$ such that every strongly $f(k)$-connected tournament contains $k$ edge-disjoint Hamiltonian cycles. This conjecture was recently proved by K\"uhn, Lapinskas, Osthus, and Patel who showed that $f(k)\leq O(k^2(\log k)^2)$ and conjectured that there is a constant $C$ such that $f(k)\leq Ck^2$. We prove this conjecture.
\end{abstract}

\section{Introduction}
A directed graph is Hamiltonian if there is a directed cycle passing through all its vertices. Hamiltonicity has a very long history in both directed and undirected graphs, and there are many results guaranteeing that a graph is Hamiltonian under certain conditions (see~\cite{BJSurvey, KOSurvey}). 

In general, it is hard to decide whether a directed graph is Hamiltonian---the problem is well known to be NP complete, even for undirected graphs. However for the special case of tournaments the problem becomes easier (a tournament is a directed graph which has exactly one edge between any pair of vertices). Here, an old result of Camion~\cite{Camion} says that a tournament is Hamiltonian if, and only if, it is strongly connected i.e. for any two vertices $x$ and $y$ there is a directed path from $x$ to $y$. Since strong-connectedness can be tested in polynomial time, this gives an efficient algorithm for testing whether a tournament is Hamiltonian.

Many results about Hamiltonicity have focused on finding several Hamiltonian cycles. Often one wants to count how many different Hamiltonian cycles there are, or to pack several edge-disjoint Hamiltonian cycles in a graph (see~\cite{KOSurvey}).
One natural condition for finding edge-disjoint Hamiltonian cycles in a tournament is \emph{strong $k$-connectedness}.
A directed graph is strongly $k$-connected if it remains strongly connected after the removal of any set of  $(k-1)$-vertices. 
Thomassen made the following conjecture about finding edge disjoint Hamiltonian cycles in a highly connected tournament.
\begin{conjecture}[Thomassen, \cite{ThomassenHamiltonian}]\label{ThomassenConjecture}
There is a  function $f(k)$ such that every strongly $f(k)$-connected tournament contains $k$ edge-disjoint Hamiltonian cycles.
\end{conjecture}

From Camion's Theorem, we have $f(1)=1$. For all larger $k$, Conjecture~\ref{ThomassenConjecture} was proved by K\"uhn, Lapinskas, Osthus, and Patel.
\begin{theorem}[K\"uhn, Lapinskas, Osthus, and Patel, \cite{KLOP}]\label{HamiltonianTheoremKLOP}
There is a constant $C$ such that every strongly $Ck^2(\log k)^2$-connected tournament contains $k$ edge-disjoint Hamiltonian cycles.
\end{theorem}

The $Ck^2(\log k)^2$ bound on the connectedness in the above theorem is close to best possible. Indeed K\"uhn, Lapinskas, Osthus, and Patel constructed tournaments which are strongly $(k-1)^2/4$-connected, but have no  $k$ edge-disjoint Hamiltonian cycles \cite{KLOP}. They conjectured that the $\log k$ factors in Theorem~\ref{HamiltonianTheoremKLOP} were unnecessary and a $Ck^2$ bound on the connectivity should suffice.
\begin{conjecture}[K\"uhn, Lapinskas, Osthus, and Patel, \cite{KLOP}] \label{HamiltonianConjectureKLOP}
There is a constant $C$ such that every strongly $Ck^2$-connected tournament contains $k$ edge-disjoint Hamiltonian Cycles.
\end{conjecture}

 The main result of this paper is a proof of this conjecture.
\begin{theorem}\label{EdgeDisjointHamiltonianCycles}
There is a constant $C$ such that every strongly $Ck^2$-connected tournament contains $k$ edge-disjoint Hamiltonian Cycles.
\end{theorem}
This theorem is proved using the method of \emph{linkage structures in tournaments}. This technique was introduced in \cite{KLOP} during the proof of Theorem~\ref{HamiltonianTheoremKLOP}. Since then  the technique has found other applications in \cite{KOT, PokrovskiyLinking} to prove results about highly connected tournaments.
The following is an informal definition of what a linkage structure is
\begin{quote}
\emph{A linkage structure $L$ in a tournament $T$, is a small subset of $V(T)$ with the property that for many pairs of vertices $x,y$ outside $L$, there is a path from $x$ to $y$ most of whose vertices are contained in $L$.}
\end{quote}
This definition is purposefully vague in order to include all previously used linkage structures. Since linkage structures arose with specific applications in mind, the exact meaning of ``small,'' ``many,'' and ``most'' in the above definition varies depending on what application one is looking at.  In applications, one first proves an intermediate result which shows that every highly connected tournament contains many disjoint linkage structures. Then these linkage structures are used to build whatever object one is looking for in the tournament (in our case Hamiltonian cycles).

The structure of this paper is as follows. In the next section, we state what properties our linkage structures have, and use them to deduce Theorem~\ref{EdgeDisjointHamiltonianCycles}.
In Section~3, we define our linkage structures (which we call ``linkers'') and derive their properties.
Finally, in Section~4 we give some concluding remarks and open problems. 

\section{Finding Hamiltonian cycles using linkage structures}
A directed graph is Hamiltonian connected if for any pair of vertices $x$ and $y$, it contains a Hamiltonian path from $x$ to $y$.
The following is a version of a theorem of K\"uhn, Osthus, and Townend. It is perhaps the simplest example of linkage structures to state.
\begin{theorem}[K\"uhn, Osthus, and Townend, \cite{KOT}] \label{LinkageStructuresKOT}
All strongly $10^{16}k^3\log(k^2)$-connected tournaments contain $k$ vertex-disjoint sets $L_1, \dots, L_k$ such  that
\begin{itemize}
\item $|L_i|\leq |T|/100k$.
\item For any $S\subseteq T\setminus (L_1\cup \dots\cup L_k)$, the subtournament on $L_i\cup S$ is Hamiltonian connected for every $i$.
\end{itemize}
\end{theorem}
This theorem is obtained from combining Theorem~1.5 from~\cite{KOT} with a theorem of Thomassen that every strongly $4$-connected tournament is Hamiltonian connected~\cite{ThomassenHamiltonianConnected}.

Comparing this theorem with the informal definition of linkage structures given in the introduction, we see that for any pair of vertices $x,y$ outside of the linkage structures $L_1, \dots, L_k$, there is a path from $x$ to $y$, \emph{all} of whose internal vertices are contained in any one of the linkage structures $L_i$.

It is easy to see how Theorem~\ref{LinkageStructuresKOT} might be useful in proving results about Hamiltonicity of tournaments. Indeed suppose that we have sets $L_1, \dots, L_k$ as in Theorem~\ref{LinkageStructuresKOT}. Then for any partition of $T\setminus (L_1\cup \dots\cup L_k)$ into $k$ paths $P_1, \dots, P_k$, there is a Hamiltonian cycle in $T$ containing  $P_1, \dots, P_k$. Indeed this cycle is obtained by successively considering pairs of paths $P_i$ and $P_{i+1 \pmod{k}}$. If $x$ and $y$ are the start and end of $P_{i+1 \pmod{k}}$ and $P_i$ respectively, Theorem~\ref{LinkageStructuresKOT} implies that there is a Hamiltonian  path from $x$ to $y$ in $L_i+x+y$. This Hamiltonian path is used to join $P_i$ to $P_{i+1 \pmod{k}}$ using all the vertices of $L_i$. Repeating this for all $i=1, \dots, k$, produces the required Hamiltonian cycle.

The following is main idea of the proofs of Theorems~\ref{HamiltonianTheoremKLOP} and~\ref{EdgeDisjointHamiltonianCycles}. First we use a result similar to Theorem~\ref{LinkageStructuresKOT} to find many disjoint linkage structures in a highly connected tournament $T$. Then, we find $k$ collections of edge-disjoint paths, each collection partitioning the remaining vertices of $T$. Finally, using the linkage structures we join each collection of paths into a Hamiltonian cycle.

To find the collections of paths, we use a theorem of Gallai and Milgram.
The independence number of a directed graph is the order of the largest subset of vertices with no edges inside it.
\begin{theorem}[Gallai-Milgram, \cite{GM}]
Let $D$ be a directed graph with independence number~$k$. Then $V(D)$ can be covered by at most $k$ vertex disjoint paths.
\end{theorem}

The degree of a vertex in a directed graph is the sum of its in and out-degrees. Notice that a directed graph with minimum degree $n-k$ must have independence number at most~$k$. Therefore the above theorem has the following corollary.
\begin{corollary}\label{GallaiMilgramCorollary}
Let $D$ be a directed graph with minimum degree $\geq n- k$. Then $V(D)$ can be covered by at most $k$ vertex disjoint paths.
\end{corollary}

Repeatedly applying this corollary to a tournament $T$ produces collections of paths $\mathcal P_1, \dots, \mathcal P_k$ such that $\mathcal P_i$ consists of $2i-1$ vertex-disjoint paths which cover $V(T)$, and also for all $i\neq j$ the paths in $\mathcal P_i$ are edge-disjoint from those in $\mathcal P_j$. 
It is the paths in these collections which the linkage structures join into Hamiltonian cycles.
Assuming we need $2i-1$ linkage structures to join the $2i-1$ paths in $\mathcal P_i$ into a cycle, we would need $k^2$ linkage structures altogether. This is the source of the quadratic bound in Theorems~\ref{HamiltonianTheoremKLOP} and~\ref{EdgeDisjointHamiltonianCycles}.

Next, we formally define the properties of the linkage structures we use. We will actually define a family of several linkage structures which we call a \emph{linking family}.
\begin{definition}~\label{LinkingFamily}
For $k\geq 1$, a family $\{L_1, \dots, L_k\}$  of vertex disjoint subdigraphs of a digraph $D$ is a linking family of size $k$ in $D$ if the following holds.

 Suppose we have two vertices $x$ and $y$ outside $L_1\cup \dots\cup L_k$ and at most $100k$ vertex disjoint paths $P_1, \dots, P_m$  in $V(T)\setminus (L_1\cup\dots\cup L_k\cup \{x,y\})$.  Then there are vertex disjoint paths $P, P'_1, \dots, P'_m$ and subdigraphs $L'_1, \dots, L'_{k-1}$ such that
\begin{enumerate}[(i)]
\item $P$ is from $x$ to $y$.
\item $P\cup P'_1\cup \dots\cup P'_m\cup L'_1\cup \dots\cup L'_{k-1}$ consists of $L_1\cup \dots\cup L_k \cup P_1\cup \dots\cup P_m\cup \{x,y\}$, plus at most $6$ other vertices.
\item $P'_j$ has the same endpoints as $P_j$ for every $j$.
\item If $k\geq 2$, then $\{L'_1, \dots, L'_{k-1}\}$ is a linking family of size $k-1$ in $D$.
\end{enumerate}
\end{definition}
Part (iv) of this definition may look a bit strange since it seems to make the whole definition self-referential. However notice that the family $\{L_1, \dots, L_k\}$ has $k$ digraphs in it, whereas the family $\{L'_1, \dots, L'_{k-1}\}$ only has $k-1$. Therefore the definition is consistent since first we define a linking family of size $1$, then a linking family of size $2$ (using linking families of size $1$), then a linking family of size $3$ (using linking families of size $2$), etc. 

It is useful to compare a linking family of size $1$ to the informal definition of linkage structures in the introduction. Given a linking family $\{L\}$ of size $1$, we see that for any pair of vertices $x$, $y$ outside $L$, there is a $x$ -- $y$ path using only at most $6$ vertices outside of $L_i+x+y$. We have no control over where these extra vertices are, so they could potentially ruin the Hamiltonian cycle we are trying to build. The purpose of the paths $P_1, \dots, P_m$    is to allow us to ``protect'' certain paths from being broken by these extra $6$ vertices we might use when joining $x$ to $y$.
We remark that the paths ${P}_i$ are allowed to consist of just one vertex in the above lemma. In this case ${P}'_i={P}_i$ will hold since there is only one possible path beginning and ending at the same vertex. This phenomenon can be useful since it allows us to protect a small number of vertices $\{v_1, \dots, v_n\}$ from ever appearing in the paths ${P}, {P}'_1, \dots, {P}'_m$ or digraphs $L'_1, \dots, L'_{k-1}$ by letting ${P}_{r+1}=v_1, \dots, {P}_{r+n}=v_n$.

The following is the main technical result of this paper. It shows that every highly connected tournament contains a large linking family.
\begin{theorem} \label{LinkageStructures}
There are constants $C_1$ and $\Delta_1$ with the following property. 
Suppose that $T$ is a strongly $C_1k$-connected tournament $T$.
Then $T$ contains $k$ vertex-disjoint subdigraphs $L_1, \dots, L_k$ with maximum degree $\Delta_1$, such that for any spanning subdigraph $D\subseteq T$ with minimum degree at least $|T|-100\Delta_1 k$, any subfamily  $\hyper L\subseteq \{L_1, \dots, L_k\}$ is a linking family in $D\cup \hyper L$.
\end{theorem}
This Theorem is proved in Section~3.
In the remainder of this section, we show how Theorem~\ref{LinkageStructures} can be used to prove Theorem~\ref{EdgeDisjointHamiltonianCycles}.


First we'll need a simple lemma about linking families.
One important feature of part (ii) of Definition~\ref{LinkingFamily} is that if $P_1, \dots, P_m, L_1, \dots, L_k, x,$ and $y$ partition $V(D)$, then (ii) implies that $P, P'_1 \dots, P'_m, L'_1, \dots, L'_{k-1}$ will partition $V(D)$ also.
This allows us to obtain the following criterion for Hamiltonicity.
\begin{lemma}\label{HamiltonianLemma}
Suppose that for $k\geq 1$, the vertices of a digraph $D$ can be partitioned into $k$ paths and a linking family of size $k$. Then $D$ is Hamiltonian.
\end{lemma} 
\begin{proof}
The proof is by induction on $k$. 

The initial case is when $k=1$. In this case we have a partition of $V(D)$ into a path $Q$ and a digraph $L$ such that $\{L\}$ is a linking family. Let $y$ and $x$ be the start and end of $Q$ respectively. Let $R=Q-x-y$. Invoking the property of linking families to the linking family $\{L\}$ with the vertices $x$ and $y$, and path $R$, we obtain two paths $P$, $R'$ such that $P$ is from $x$ to $y$ and $R'$ has the same endpoints as $R$. In addition from (ii), we have that $P$ and $R'$ partition $V(D)$. Joining $P$ to $R'$ produces a Hamiltonian cycle.

Now suppose that the lemma holds for $k=k_0$. Suppose that we have a partition of $V(D)$ into $k_0+1$ paths $Q_1, \dots, Q_{k_0+1}$ and a  linking family  $\{L_1, \dots, L_{k_0+1}\}$. 
Let $y$ and $x$ be the start and end of $Q_{k_0+1}$ and $Q_{k_0}$ respectively. Let $Q_-=Q_{k_0}-x$ and $Q_+=Q_{k_0+1}-y$.
Invoking the property of linking families with vertices $x$ and $y$, and paths $Q_1, \dots, Q_{k_0-1}, Q_-, Q_+$, we obtain a path $P$ from $x$ to $y$, a new linking family $\{L'_1, \dots, L'_{k_0}\}$ and new paths $Q'_1, \dots, Q'_{k_0-1}, Q'_-, Q'_+$ with the same endpoints as the previous ones. In addition  $L'_1, \dots, L'_{k_0}$, $Q'_1, \dots, Q'_{k_0-1}, Q'_-, Q'_+$, and $P$ partition $V(D)$. Join $Q'_-$ to $P$ to $Q'_+$ in order to obtain a path $Q'_{k_0}$. Now we have a partition of $D$ into $k$ paths $Q'_1, \dots, Q'_{k_0}$ and a linking family $\{L'_1, \dots, L'_{k_0}\}$. By induction, $D$ is Hamiltonian.
\end{proof}

Combining the above lemma with Theorem~\ref{LinkageStructures} and Corollary~\ref{GallaiMilgramCorollary}, it is easy to prove Conjecture~\ref{HamiltonianConjectureKLOP}.
\begin{proof}[Proof of Theorem~\ref{EdgeDisjointHamiltonianCycles}]
Let $C_1$ and $\Delta_1$ be the constants in Theorem~\ref{LinkageStructures}, and set $C=(\Delta_1+2) C_1$.
Let $T$ be a strongly $Ck^2$-connected tournament. Apply Theorem~\ref{LinkageStructures} in order to obtain a family of $({\Delta_1}+2)k^2$ vertex-disjoint subdigraphs $\{L_{i,j}: 1\leq i\leq k, 1\leq j\leq ({\Delta_1}+2)k\}$.

Let $D_1$ be the digraph formed from $T$ by removing  the edges of the digraphs in $\{L_{i,j}: 2\leq i\leq k ,1\leq j\leq ({\Delta_1}+2)k\}$. Notice that $D_1$ has minimum degree $|T|-{\Delta_1}$. Thus, from Theorem~\ref{LinkageStructures}, the family $\{L_{1,1}, \dots, L_{1,(\Delta_1+2)k}\}$ is a linking family in $D_1$.
Apply Corollary~\ref{GallaiMilgramCorollary}  in order to cover $D_{1}\setminus \big(V(L_{1,1}) \cup\dots\cup V(L_{1, ({\Delta_1}+2)k})\big)$ by $\Delta_1$ vertex-disjoint paths. By splitting some of these paths in two we can find a partition of $D_{1}\setminus \big(V(L_{1,1}) \cup\dots\cup V(L_{1, ({\Delta_1}+2)k})\big)$ into exactly $(\Delta_1+2)k$ paths. Applying  Lemma~\ref{HamiltonianLemma} produces a Hamiltonian cycle~$C_1$ in~$D_1$. 

In general, for any $\ell$ between $2$ and $k$, let $D_{\ell}$ be the digraph formed from $T$ by removing  the edges of all the digraphs in $\{L_{i,j}: \ell+1\leq i\leq k ,1\leq j\leq ({\Delta_1}+2)k\}$ and the cycles $C_1, \dots, C_{\ell-1}$. Notice that $D_{\ell}$ has minimum degree $|T|-{\Delta_1}-2\ell$, and so Theorem~\ref{LinkageStructures} implies that the family $\{L_{\ell,1}, \dots, L_{\ell,(\Delta_1+2)k}\}$ is a linking family in $D_{\ell}$.
Apply Corollary~\ref{GallaiMilgramCorollary}  in order to cover $D_{\ell}\setminus \big(V(L_{\ell,1})\cup \dots\cup V(L_{\ell, ({\Delta_1}+2)k})\big)$ by ${\Delta_1}+2\ell$ vertex-disjoint paths. 
By splitting some of these paths in two we can find a partition of $D_{\ell}\setminus \big(V(L_{\ell,1})\cup \dots\cup V(L_{\ell, ({\Delta_1}+2)k})\big)$ into exactly $(\Delta_1+2)k$ paths.
Applying Lemma~\ref{HamiltonianLemma} produces a 
Hamiltonian cycle $C_{\ell}$ in $D_{\ell}$. 

This gives us the required edge-disjoint Hamiltonian cycles $C_1, \dots, C_{k}$.
\end{proof}

\section{Linkers}
The goal of this section is to prove Theorem~\ref{LinkageStructures}. 
We do this by constructing digraphs which we call \emph{linkers}, such that any family of linkers is a linking family.

Before we can even define linkers, we first need to set up some notation and construct two kinds of gadgets which we call dominators and connectors. 
In the next section we define some notation and prove some auxiliary lemmas about tournaments. In Sections 3.2 and 3.3, we define dominators and connectors. Then in Section 3.4 we define linkers. In  Section 3.5 we show that every highly connected tournament contains many disjoint linkers.  In Section 3.6 we derive the properties of linkers which we will need. Then in Section 3.7 we put everything together and prove Theorem~\ref{LinkageStructures}.

\subsection{Preliminaries}
A directed path $P$ is a sequence of vertices $v_1, v_2,\dots, v_k$ in a directed graph such that $v_i v_{i+1}$ is an edge for all $i=1, \dots, k-1$. All paths in this paper are directed paths. The vertex $v_1$ is called the \emph{start} of $P$, and $v_k$ the \emph{end} of $P$. The \emph{length} of $P$ is the number of edges it has which is $|P|-1$. The vertices $v_{2}, \dots, v_{k-1}$ are the \emph{internal vertices} of $P$. Two paths are said to be internally disjoint if their internal vertices are distinct.

The \emph{out-neighbourhood} of a vertex $v$ in a directed graph, denoted $N^+(v)$ is the set of vertices $u$ for which $vu$ is an edge. Similarly, the \emph{in-neighbourhood}, denoted $N^-(v)$ is the set of vertices $u$ for which $uv$ is an edge.
The \emph{out-degree} of $v$ is $d^+(v)=|N^+(v)|$, and the \emph{in-degree} of $v$ is $d^-(v)=|N^-(v)|$. A useful fact is that every tournament $T$ has a vertex of out-degree at least $(|T|-1)/2$, and a vertex of in-degree at least $(|T|-1)/2$. To see this, notice that since $T$ has $\binom{|T|}2$ edges, its average in and out-degrees are both $(|T|-1)/2$.

We'll need the following definition.
\begin{definition}
A vertex $v$ in a tournament $T$ has large out-degree if there are less than $|T|/25$ vertices $u\in T$ satisfying $d^+(u)> d^+(v)$
\end{definition}
Vertices with \emph{large in-degree} are defined similarly---a vertex has large in-degree in $T$ if there are less than $|T|/25$ vertices $u\in T$ satisfying $d^-(u)> d^-(v)$. Notice that every tournament $T$ contains at least $|T|/25$ vertices of large out-degree, and $|T|/25$ vertices of large in-degree.

Recall that every tournament $T$ has a vertex of out-degree at least $(|T|-1)/2$. By repeatedly pulling out maximum out-degree vertices, this implies that every tournament $T$ contains at least $k$ vertices of out-degree at least $(|T|-k)/2$.
Therefore, if $v$ has large out-degree in $T$, then it must satisfy $d^+(v)\geq 12|T|/25$. 

The important feature of vertices of large in-degrees and out-degrees is that for any pair of vertices one of which has large out-degree, and the other large in-degree, there are many short paths between them.
\begin{lemma}\label{LargeDegreeShortPath}
Suppose that $u$ has large out-degree in $T$ and $v$ has large in-degree in $T$. Then there are at least $|T|/25$ internally vertex-disjoint  paths from $u$ to $v$ in $T$, each of length at most $3$.
\end{lemma}
\begin{proof}
Let $I=N^+(u)\cap N^-(v)$, $U=\big(N^+(u)+u\big)\setminus \big(N^-(v)+v\big)$, $V=\big(N^-(v)+v\big)\setminus \big(N^+(u)+u\big)$, and $M$ a maximum matching of edges directed from $U$ to $V$.

Notice that there are exactly $|I|+e(M)$ paths of length $\leq 3$ from $u$ to $v$, and so if $|I|+e(M)\geq |T|/25$ holds, then we are done. So, suppose for the sake of contradiction that we have $|I|+e(M)< |T|/25$.

Recall that since $u$ has large out-degree we have $d^+(u)\geq 12|T|/25$. This implies
$$|U\setminus M|=|N^+(u)+u-v|-|I|-e(M)\geq |N^+(u)|-|T|/25\geq 11|T|/25.$$
Similarly we have $|V\setminus M|\geq 11|T|/25.$ Since $N^+(u)\subseteq T\setminus (V+u)$ we obtain $d^+(u)\leq 14|T|/25+1.$

Since $M$ is maximal, all the edges between $U\setminus M$ and $V\setminus M$ are directed from $V$ to $U$.
Therefore the $|T|/25$ vertices of largest out-degree in $V$ all have out-degree at least $|U\setminus M|+(|V\setminus M|-|T|/25)/2\geq 16|T|/25$. Since $d^+(u)< 14|T|/25+1$, this contradicts $u$ having large out-degree.
\end{proof}

A tournament $T$ is transitive if for any three vertices $x,y,z\in V(T)$, if $xy$ and $yz$ are both edges, then $xz$ is also an edge. It's easy to see that a tournament is transitive exactly when it has an ordering $(v_1, v_2,\dots, v_k)$ of $V(T)$ such that the edges of $T$ are $\{v_iv_j: i<j\}$. We say that $v_1$ is the \emph{tail} of $T$, and $v_k$ is the \emph{head} of $T$.

A simple, but very important fact is that every tournament contains a large transitive subtournament.
\begin{lemma}\label{TransitiveSubtournament}
Every tournament $T$ contains a transitive subtournament on at least $\log_2 |T|$ vertices.
\end{lemma}
This lemma is proved by choosing the vertex sequence $(v_1, \dots, v_k)$ of the transitive tournament recursively, by letting $v_i$ be a maximum out-degree vertex in $\bigcap_{j=1}^{i-1} N^+(v_i)$.

A set of vertices $S$ in-dominates another set $B$, if for every $b\in B\setminus S$, there is some $s\in S$ such that $bs$ is an edge. Notice that by this definition, a set in-dominates itself. A \emph{in-dominating set} in a tournament $T$ is any set $S$ which in-dominates $V(T)$. Notice that by repeatedly pulling out vertices of largest in-degree and their in-neighbourhoods from $T$, we can find an in-dominating set of order at most $\lceil\log_2 |T|\rceil$. For our purposes we'll study sets which are constructed by pulling out some fixed number of vertices by this process.

\begin{definition}
We say that a sequence $(v_1,v_2,\dots, v_k)$ of vertices of a tournament $T$ is a partial greedy in-dominating set if $v_1$ is a maximum in-degree vertex in $T$, and for each $i$, $v_i$ is a maximum in-degree vertex in the subtournament of $T$ on $N^+(v_1)\cap N^+(v_2)\cap \dots \cup N^+(v_{i-1})$.
\end{definition}
Partial greedy out-dominating sets are defined similarly, by letting $v_i$ be a maximum out-degree vertex in $N^-(v_1)\cap N^-(v_2)\cap \dots \cap N^-(v_{i-1})$ at each step.

Notice that every partial greedy in-dominating set is a transitive tournament with head $v_k$ and tail $v_1$.

For small $k$, partial greedy in-dominating sets do not necessarily dominate all the vertices in a tournament. A crucial property of partial greedy in-dominating sets is that the vertices they don't dominate have large out-degree.
The following is a version of a lemma appearing in~\cite{KLOP}.
\begin{lemma}\label{GreedyDominatingSet}
Let $(v_1,v_2,\dots, v_k)$ be a partial greedy in-dominating set in a tournament~$T$. Let $E$ be the set of vertices which are not in-dominated by $A$. Then every $u\in E$ satisfies $d^+(u)\geq 2^{k-1}|E|.$
\end{lemma}
\begin{proof}
The proof is by induction on $k$. The initial case is when $k=1$. In this case we have $E=N^+(v_1)$ where $v_1$ is a maximum in-degree vertex in $T$. For any $u\in E$, we must have $d^-(u)\leq d^-(v_1)=|T\setminus E- v_1|=|T|-|E|-1$. Therefore we have $d^+(u)=|T\setminus N^-(u)-u|= |T|-d^-(u)-1\geq |E|$ as required.

Now suppose that the lemma holds for $k=k_0$. Let $(v_1, \dots, v_{k_0+1})$ be a partial greedy in-dominating set in $T$, and let $E_0=N^+(v_1)\cap \dots \cap N^+(v_{k_0})$. 
By induction we have $d^+(u)\geq 2^{k_0-1}|E_0|$ for every $u\in E_0$. 
By definition $v_{k_0+1}$ is a maximum in-degree vertex in $E_0$. Let $E=E_0\cap N^+(v_{k_0+1})$ be the set of vertices not in-dominated by $(v_1, \dots, v_{k_0+1})$. 
Since $v_{k_0+1}$ is a maximum in-degree vertex in $E_0$, we have $|N^-(v_{k_0+1})\cap E_0|\geq (|E_0|-1)/2$ which implies $|E|=|E_0|-|(N^-(v_{k_0+1})+v_{k_0+1})\cap E_0|\leq |E_0|/2$.
Combining this with the inductive hypothesis, we obtain $d^+(u)\geq 2^{k_0-1}|E_0|\geq 2^{k_0}|E|$, completing the proof. 
\end{proof}

\subsection{Dominators}
In order to construct our linking structures, we will need special sets of vertices which we call \emph{dominators}. Informally, a dominator behaves like a partial greedy dominating set, but with some ``extra'' vertices which can be removed without ruining the domination.

\begin{definition}
A $(m,M,p)$-indominator $D^-$ in a tournament $T$ is a $5$-tuple $(A^1$, $A^2$, $A^3$, $A^4$, $E^-)$ of sets of vertices  in $T$ with the following properties.
\begin{enumerate}[(D1)]
\item $A^1$, $A^2$, $A^3$, and $A^4$ are all disjoint.
\item \label{Indom:Transitive} For $i=1,2,3$ the tournament on $A^i \cup A^{i+1}$ is transitive with tail in $A^i$ and head in~$A^{i+1}$. 
\item $|A^2|=|A^3|=m$.
\item $|A^1|=|A^4|=M$.
\item $A^2\cup A^3$ in-dominates $T\setminus (A^1\cup A^2\cup A^3\cup A^4\cup E^-)$. 
\item  \label{Indom:EExpansion} $d^+(v)\geq p|E^-|$ for every $v\in E^-$.
\end{enumerate}
\end{definition}
We call $E^-$ the \emph{uncovered} set of the indominator. The \emph{vertex set} of the indominator, denoted $V(D^-)$ is the set $A^1\cup A^2\cup A^3\cup A^4$.

We say that $D^-$ is an $(m,M,p)$-indominator in $T$ \emph{with exceptional set} $X$ if $D^-$ is an indominator in $(T\setminus X)\cup V(D^-)$. This terminology will be convenient since we will sometimes have many indominators in a single tournament $T$, all of which have different exceptional sets. 

An outdominator is defined to be an indominator in the tournament formed from $T$ by reversing all arcs. For convenience we list its properties here.
\begin{definition}
A $(m,M,p)$-outdominator $D^+$ in a tournament $T$ is a $5$-tuple $(B^1$, $B^2$, $B^3$, $B^4$, $E^+)$ of sets of vertices  in $T$ with the following properties.
\begin{enumerate}[(D1)]
\item $B^1$, $B^2$, $B^3$, and $B^4$ are all disjoint.
\item For $i=1,2,3$ the tournament on $B^i \cup B^{i+1}$ is transitive with head in $B^i$ and tail in~$B^{i+1}$.
\item $|B^2|=|B^3|=m$.
\item $|B^1|=|B^4|=M$.
\item  $B^2\cup B^3$ out-dominates $T\setminus (B^1\cup B^2\cup B^3\cup B^4\cup E^+)$. 
\item  $d^-(v)\geq p|E^+|$ for every $v\in E^+$.
\end{enumerate}
\end{definition}

When dealing with indominators, they will always be labelled by ``$D^-$'' (possibly with some subscript), their four sets of vertices will always be labelled by ``$A^1, \dots, A^4$'', and the set of uncovered vertices will be labelled ``$E^-$''. Similarly outdominators will always be labelled as in their definition. Exceptional sets of vertices will always be labelled by the letter ``$X$''.
The tail of the transitive tournament on $A^1$ in  an indominator $D^-$ is called the the \emph{tail of }$D^-$, and the head of $A^4$ is the \emph{head of }$D^-$. Similarly in an outdominator $D^+$, the head of $B^1$ and the tail of $B^4$ are called the \emph{head and tail of }$D^+$ respectively.

The following lemma is an intermediate step we need in order to construct dominators.
\begin{lemma}\label{PredominatorExistence}
For any numbers $m$, $M$, $L$, and $p$ with $L\geq 2^{m+M}$ and $p\leq 2^{m-1}$, the following holds.
If $T$ is a tournament with $|T|\geq L$, then there are sets of vertices $A,B, E^-,X \subseteq V(T)$ with the following properties.
\begin{enumerate}[(i)]
\item $A\cup B$ is a transitive tournament with its tail in $A$ and its head in $B$
\item $|A|=m, |B|=M$.
\item $A$ in-dominates $T\setminus (E^-\cup X)$.
\item $|X|\leq L$.
\item $d^+(u)\geq p|E^-|$ for every vertex $u \in E^-$.
\end{enumerate}
\end{lemma}
\begin{proof}
We choose a set $X\subset V(T)$ and a disjoint sequence of vertices $v_1, v_2, \dots, v_k$  with the following properties
\begin{enumerate}[(a)]
\item $|X|\leq L$.
\item All edges between $\{v_1, v_2, \dots, v_k\}$ and $X$ are oriented from $\{v_1, v_2, \dots, v_k\}$ to $X$.
\item $(v_1, v_2, \dots, v_k)$ is a partial greedy in-dominating set in $T\setminus X$.
\item $|X|$ is as large as possible, whilst keeping (a) -- (c) true.
\item The value $k$ is as large as possible, whilst keeping (a) -- (d) true.
\end{enumerate} 
To see that such a choice is possible, notice that choosing $X=\emptyset$ and $(v_1, \dots, v_k)$ to be any partial greedy in-dominating set in $T$ gives a sequence satisfying (a) -- (c).  Therefore it is also possible to choose $X$ and  $(v_1, \dots, v_k)$ to be maximal in the sense of (d) and (e). Notice that condition (c) implies that $(v_1, \dots, v_k)$ is a transitive tournament with head $v_k$, and tail $v_1$. There are two cases depending on whether $k\geq m+M$ or not.

Suppose that $k\geq m+M$. Let $A=\{v_1,\dots, v_m\}$, $B=\{v_{m+1}, \dots, v_{m+M}\}$, $E^-=N^+(v_1)\cap\dots\cap N^+(v_{m})\setminus X$. Notice that conditions (i) -- (iv) hold with this choice. Condition (v) follows from  Lemma~\ref{GreedyDominatingSet} and $p\leq 2^{m-1}$.

Suppose that $k< m+M$. By maximality of $k$ in (e), the set ${v_1, \dots, v_m}$  in-dominates $T\setminus X$.  We must also have $|X|=L$, since otherwise adding $x_k$ to $X$ would produce a larger set satisfying (a) -- (c), contradicting maximality of $X$ in (d). Therefore, since $L\geq 2^{m+M}$, Lemma~\ref{TransitiveSubtournament} implies that there is a transitive subtournament $X_T$ of $X$ order $m+M-k$. 
Let the vertex sequence of $X_T$ be $(v_{k+1}, \dots, v_{m+M})$. 
Let $X'=X\setminus X_T$, $E^-=\emptyset$,  $A=\{v_1, \dots, v_{m}\}$, and $B=\{v_{m+1}, \dots, v_{M}\}$.
All the conditions (i) -- (v) are immediate with this choice of sets.
\end{proof}

The following lemma guarantees the existence of dominators in tournaments.
\begin{lemma}\label{DominatorExistence}
For any numbers $m$, $M$, $L$, and $p$ with $L\geq 2^{m+M}+m+M$ and $p\leq 2^{m-1}$, the following holds. Let $T$ be a tournament on at least $25(2^{2m+2M})$ vertices and $Y\subseteq V(T)$ with $|Y|\leq |T|/25-2^{2m+2M}$.  Then, $T$ contains a $(m,M,p)$-indominator $D^-=(A^1, A^2, A^3, A^4, E^-)$  with an exceptional set $X$ such that $Y\cap V(D^-)=\emptyset$, $Y\subseteq X$, $|X|\leq L+|Y|$, and all vertices of $A^1$ have large in-degree in $T$.
\end{lemma}
\begin{proof}
Notice that there are at least $2^{2m+2M}$ vertices in $T\setminus Y$ of large in-degree. Therefore, by Lemma~\ref{TransitiveSubtournament}, we can choose a transitive subtournament $S\subseteq T\setminus Y$ of $2m+2M$ vertices with large indegree.  Let $S_1$ be the first $M$ vertices of $S$, $S_2$ the next $m$ vertices, $S_3$ the next $m$ vertices, and $S_4$ the last $M$ vertices. Let $T'=T\setminus (Y\cup S_1\cup S_2\cup S_3\cup S_4\cup N^-(S_2))$.  

If $|T'|\leq L$, then the lemma follows by choosing $A^i=S_i$ for $i=1,$ $2,$ $3,$ $4$, $E^-=\emptyset$, and $X=V(T')\cup Y$.

If $|T'|\geq L$, then we apply Lemma~\ref{PredominatorExistence} to $T'$ with the parameters $m$, $M$, $p$, and $L'=L-M-m$. This gives us sets $A,$ $B,$ $E^-,$ and $X$ as in Lemma~\ref{PredominatorExistence} satisfying $|A|=m$, $|B|=M$, and $|X|\leq L-M-m$. Then we let $A^1=S_1$, $A^2=S_2$, $A^3=A$, $A^4=B$, and $X'=X\cup Y\cup S_3\cup S_4$. With this definition $(A^1,A^2, A^3,A^4, E^-)$, is an indominator in $T$ with exceptional set $X'$. Indeed, conditions (D1) -- (D5) are immediate, and part (v) of Lemma~\ref{PredominatorExistence} implies that (D6) holds.
\end{proof}

By reversing arcs, we obtain the following version of Lemma~\ref{DominatorExistence} for outdominators.
\begin{lemma}\label{OutdominatorExistence}
For any numbers $m$, $M$, $L$, and $p$ with $L\geq 2^{m+M}+m+M$ and $p\leq 2^{m-1}$, the following holds. Let $T$ be a tournament on at least $25(2^{2m+2M})$ vertices and $Y\subseteq V(T)$ with $|Y|\leq |T|/25-2^{2m+2M}$.  Then, $T$ contains a $(m,M,p)$-outdominator $D^+=(B^1, B^2, B^3, B^4, E^+)$  with an exceptional set $X$ such that $Y\cap V(D)=\emptyset$, $Y\subseteq X$, $|X|\leq L+|Y|$, and all vertices of $B^1$ have large out-degree in $T$.
\end{lemma}

Given an indominator $D^-$ in a tournament $T$, we will sometimes want to modify $T$, and still know that $D^-$ is an indominator in the modified tournament. If $D^-$ has exceptional set $X$, then from the definition of ``exceptional set,'' we see that removing any vertices of $X\setminus V(D^-)$ from $T$ will preserve $D^-$ being an indominator. Similarly, we can remove vertices of $X\cap A_1$ and $X\cap A_4$ to obtain a new indominator. Corresponding results hold for outdominators as well.

Given a dominator $D$ with exceptional set $X$, we will sometimes want to increase the size of $X$ and still know that $D$ is a dominator with the larger exceptional set. The following lemma allows us to do this under the assumption that $T$ has large degree.
\begin{lemma}\label{DominatorRobustness}
Let $T$ be a tournament of minimum  out-degree $\delta^+(T)$, and $D^-$ an $(m,M,p)$-indominator in $T$ with exceptional set $X$. For any $Y$ satisfying $X\subseteq Y$ and $2|Y|\leq \delta^+(T)$, $D^-$ is an $(m,M,p/2)$-indominator in $T$ with exceptional set $Y$.
\end{lemma}
\begin{proof}
The only part of the definition of an indominator which needs checking is (D\ref{Indom:EExpansion}).
Let $E^-$ be the set of uncovered vertices of $D^-$ in $T\setminus X$, and $v$ be a vertex in $E^-$. We need to show that $|N^+(v)\setminus Y| \geq p|E^-\setminus Y|/2$.

Since $D^-$ is an $(m,M,p)$-indominator in $T$ with exceptional set $X$, we have $d^+(v)\geq p|E^-|$. We also have $d^+(v)\geq 2|Y|$. Averaging these gives $d^+(v)\geq p|E^-|/2+|Y|$. This implies the result
$$|N^+(v)\setminus Y| \geq d^+(v)-|Y|\geq p|E^-|/2\geq p|E^-\setminus Y|/2.$$
\end{proof}

By reversing arcs in the above lemma, we obtain the following version of it for outdominators
\begin{lemma}\label{OutdominatorRobustness}
Let $T$ be a tournament of minimum  in-degree $\delta^-(T)$, and $D^+$ an $(m,M,p)$-outdominator in $T$ with exceptional set $X$. For any $Y$ satisfying $X\subseteq Y$, and $2|Y|\leq \delta^-(T)$, $D^+$ is an $(m,M,p/2)$-outdominator in $T$ with exceptional set $Y$.
\end{lemma}


\subsection{Connectors}
In order to construct our linking structures, we will need special gadgets which we call ``connectors''. Informally, a connector is a small set of vertices together with two coverings of it---one by four paths, and one by five.

\begin{definition}
A connector is any digraph $C$ on at most $40$ vertices and containing distinct vertices $x_1, \dots, x_5, y_1, \dots, y_5$ with the following property.  
For $n\in \{4,5\}$, there are vertex disjoint paths $P_1, \dots, P_n$ such $P_i$ is from $x_i$ to $y_i$, and $V(C)=V(P_1)\cup \dots\cup V(P_n)$.
\end{definition}
The vertices $x_1, \dots, x_5$ are the \emph{sources} of the connector and the vertices  $y_1, \dots, y_5$ are the \emph{sinks} of the connector.

One example of a connector is a transitive tournament $T$ on $10$ vertices, with vertex sequence $x_1, \dots, x_5, y_1, \dots, y_5$. Its easy to see that for $n= 4$ or $5$, we can find $n$ disjoint $x_i$~--~$y_i$ paths covering $T$.
For our purposes, we'll need slightly more complicated connectors.
The following lemma allows us to find a connector with prescribed sources and sinks under certain conditions.

\begin{lemma}\label{ConnectorExistence}
There is a constant $N=100 \cdot 2^{2^{2^{2^{10}}}}$ such that the following holds.
Let $T$ be a tournament on at least $200N$ vertices, $Y$ a set of vertices in $T$ with $|Y|\leq |T|/50$, and $\{x_1, \dots, x_{N}, y_1, \dots, y_{N}\}$ a set of $2N$ vertices in $T\setminus Y$ such that $x_1, \dots, x_{N}$ have large out-degree in $T$, and $y_1, \dots, y_{N}$ have large in-degree in $T$.
Then there is a connector $C$ contained in $T\setminus Y$, such that the sources of $C$ are in $\{x_1, \dots, x_{N}\}$, and the sinks in  $\{y_1, \dots, y_{N}\}$.
\end{lemma}
\begin{proof}
Lemma~\ref{LargeDegreeShortPath} implies that for any $i=1,\dots,N$ there are at least $|T|/25$ internally vertex disjoint paths of length at most $3$ from $x_i$ to $y_i$.
 Since $|Y|\leq |T|/50$, there are at least $|T|/50$ such paths avoiding $Y$. 
Therefore, using $|T|\geq 200N$, we can choose vertex disjoint paths $P_1, \dots, P_{N}$ of length at most $3$ in $T\setminus Y$, such that $P_i$ is from $x_i$ to $y_i$.

Notice that at least a third of these paths must have the same length. Without loss of generality we can assume that the paths $P_1, \dots, P_{N/3}$ all have ${\ell}$ vertices.

For $i=1, \dots, N/3$, let $p_i^{1},\dots, p_i^{{\ell}}$ be the vertex sequence of $P_i$ (so we have $p_i^1=x_i$ and $p_i^{\ell}=y_i$). By Lemma~\ref{TransitiveSubtournament}, there is some subset $I_1\subseteq [N/3]$ with $|I_1|\geq \log_2 N/3$ such that the subtournament on $\{p_i^1:i\in I_1\}$ is transitive. Applying Lemma~\ref{TransitiveSubtournament} again, we find some subset $I_2\subseteq I_1$ with $|I_2|\geq\log_2 I_1$ such that the subtournaments on $\{p_i^2:i\in I_2\}$  and $\{p_i^1:i\in I_2\}$ are both transitive. Applying Lemma~\ref{TransitiveSubtournament} ${\ell}-2$ more times we obtain a subset $I_{\ell}\subseteq I_2$ with $|I_{\ell}|\geq\log_2 \log_2 \log_2 \log_2 N/3\geq 10$ such that the subtournaments on $\{p_i^j:i\in I_{\ell}\}$  are  transitive for $j=1,\dots,{\ell}$.
Without loss of generality, we can suppose that $I_{\ell}$ contains the set $\{1,\dots, 10\}$.

For each $j=1, \dots, {\ell}$, we define a subtournament $T_j$, vertices $h_j$ and $t_j$, and two paths $P_h^j$ and $P_t^j$ as follows:
Let $T_1$ be the subtournament of $T$ on vertices $V(P_1)\cup\dots\cup V(P_{10})$. Let $h_1$ and $t_1$ be the head and tails respectively of the transitive tournament on $\{p_i^1:1\leq i\leq 10\}$. Let $P_h^1$ and $P^1_t$ be the $P_i$-paths containing $h_1$ and  $t_1$ respectively. 
Then, for $j=2,\dots, {\ell}$, let $T_j= T_{j-1}\setminus (P_h^{j-1}\cup P_t^{j-1})$. Let $h_j$ and $t_j$ be the head and tails respectively of the transitive tournament on $T_j\cap \{p_i^j:1\leq i\leq 10\}$. Let $P_h^j$ and $P^j_t$ be the $P_i$-paths containing $h_j$ and $t_j$ respectively. 

For $j=1,\dots, {\ell}$, let $P'^j_h$ be the final segment of the path $P_h^j$ starting from $h_j$, and let $P'^j_t$ be the initial segment of the path $P_t^j$ ending at $t_j$.

We can now define the connector $C$. Let $C$ the the subtournament of $T$ on the vertices $\left(T_1\setminus \bigcup_{j=1}^{\ell} P_h^j\cup P_t^j\right)\cup \left( \bigcup_{j=1}^{\ell} P'^j_h\cup P'^j_t\right)$.  In other words $C$ is the tournament on the vertices of $T_1$ with all the paths $P_h^j$ and  $P_t^j$ removed, but then with the initial and final segments $P'^j_h$ and  $P'^j_t$ added back in.  Notice that we have $|C|\leq 40$. Let $y'_1, \dots, y'_{\ell}$ be the ends of the paths $P'^1_h, \dots, P'^{\ell}_h$. Let $x'_1, \dots, x'_{\ell}$ be the starts of the paths $P'^1_t, \dots, P'^{\ell}_t$. Notice that since ${\ell}\leq 4$, there must be at least $5-{\ell}$ paths in $\{P_1,\dots ,P_{10}\}$ which are distinct from $P_h'^1\dots,P_h'^{\ell}, P_t'^1\dots,P_t'^{\ell}$.  Let $x'_{{\ell}+1},\dots, x'_5$ be the starts of any choice of such paths. Let $y'_{{\ell}+1},\dots, y'_5$ be the ends of the paths containing   $x'_{{\ell}+1},\dots, x'_5$. 

We claim that $C$ is a connector with sources $x'_1, \dots, x'_5$ and sinks $y'_1, \dots, y'_5$.  To see this, let $n = 4$ or $5$. For $i={\ell}+1, {\ell}+2, \dots, n$, let $P'_{i}$ be the path between $x'_i$ and $y'_i$ (which is one of the paths in $\{P_1,\dots ,P_{10}\}$). For $i=1,\dots, {\ell}$, let $R_i$ be a path from $t_i$ to $h_i$ consisting of all the vertices in $V(C)\cap\{p_t^i:1\leq t\leq 10\}\setminus \big(V(P'_{{\ell}+1})\cup\dots\cup V(P'_{n})\big)$ (such a path exists because $V(C)\cap \{p_t^i:1\leq t\leq 10\}$ is a transitive tournament with head $h_i$ and tail $t_i$). For $i=1,\dots, {\ell}$, let $P'_i$ be the path formed by joining $P'^i_h$ to $R_i$ to $P'^i_t$. Now, we have that for each $i=1,\dots, n$, $P'_i$ goes from $x'_i$ to $y'_i$, and $V(C)=V(P'_1)\cup \dots\cup V(P'_n)$ as required.
\end{proof}

\subsection{Definition of linkers}
Here we define our linkage structures.
\begin{definition}\label{LinkerDefinition}
A $t$-linker $L$ in $T$ consists of $t$ $(8,8,8)$-indominators $D^-_1,$ $\dots,$ $D^-_t$ with $D^-_i=(A_i^1,$ $A_i^2,$ $A_i^3,$ $A_i^4,$ $E^-_i)$, $t$ $(8,8,8)$-outdominators $D^+_1, \dots, D^+_t$ with $D^+_i=(B_i^1,$ $B_i^2,$  $B_i^3,$ $B_i^4,$ $E^+_i)$, $t$ connectors $C_1,$ $\dots$ $C_{t}$,  $5t$ directed paths $Q_1, \dots, Q_{5t}$, and a set $X$ which have the following properties.
\begin{enumerate} [(L1)]
\item  The indominators $D^-_1, \dots, D^-_t$, outdominators $D^+_1, \dots, D^+_t$, paths $Q_1, \dots, Q_{5t},$ and connectors,  $C_1, \dotsm C_{t}$ are all vertex disjoint.
\item The indominators $D^-_1, \dots, D^-_t$ and outdominators $D^+_1, \dots, D^+_t$ all have the common exceptional set $X$. We have $V(D^-_i), V(D^+_i), V(C_i) \subseteq X$ for all $i=1, \dots, t$.  
\item We have $|E^-_1|\geq |E^-_2|\geq \dots \geq |E^-_t|$ and $|E^+_1|\geq |E^+_2|\geq \dots \geq |E^+_t|$.
\item \label{UncoveredBiggerLinker} Either $E^-_t\geq E^+_1$ or $E^+_t\geq E^-_1$ holds.
\item \label{LinkerEdges} For all $i$ and $j$, the following directed edges are present:
\begin{itemize}
\item Every edge from the sinks of $C_j$ to any vertex in $A_i^1$.
\item Every edge from the any vertex in $B_i^1$ to the sources of $C_j$.
\item Every edge from the any vertex in $A_i^4$ to the start of $Q_j$.
\item Every edge from the end of $Q_j$ to any vertex in $B_i^4$.
\end{itemize}
\end{enumerate}
\end{definition}
The \emph{vertices of $L$} are $V(L)=\left(\bigcup_{i=1}^{5t}V(Q_i) \right)\cup \left(\bigcup_{i}^t V(D^-_i)\cup V(D^+_i)\cup V(C_i)\right)$.
The vertices inside $\bigcup_{i}^t V(D^-_i)\cup V(D^+_i)\cup V(C_i)$ are called the \emph{essential vertices} of the $t$-linker. The vertices in $Q_1, \dots, Q_{5t}$ are called the \emph{path vertices} of the $t$-linker.
The \emph{edges of $L$} are all the edges contained in the dominators $D^-_i$, $D^+_i$, connectors $C_i$, paths $Q_i$, as well as all the edges mentioned in (L\ref{LinkerEdges}).

The set $X$ is called the exceptional set of the $t$-linker. 
Notice that if $L$ is a linker in $T$ with exceptional set $X$, then removing any vertices of $X\setminus V(L)$ from $T$ produces a new tournament $T'$ where $L$ is still a linker with exceptional set $X\cap T'$.

It is worth noticing that if $L$ is a $t$-linker in a tournament $T$, then it will also be a $t$-linker in the tournament $T^{op}$ produced from $T$ by reversing all arcs (where we also exchange the roles of the indominators and outdominators in $L$). This will be useful since it allows us to assume that $E^-_t\geq E^+_1$  occurs in (L\ref{UncoveredBiggerLinker}), as long as we are only working with one linker in a tournament.

See Figure~\ref{LinkerFigure} for an illustration of a linker.

\begin{figure}
  \centering
     \includegraphics[width=1.0\textwidth]{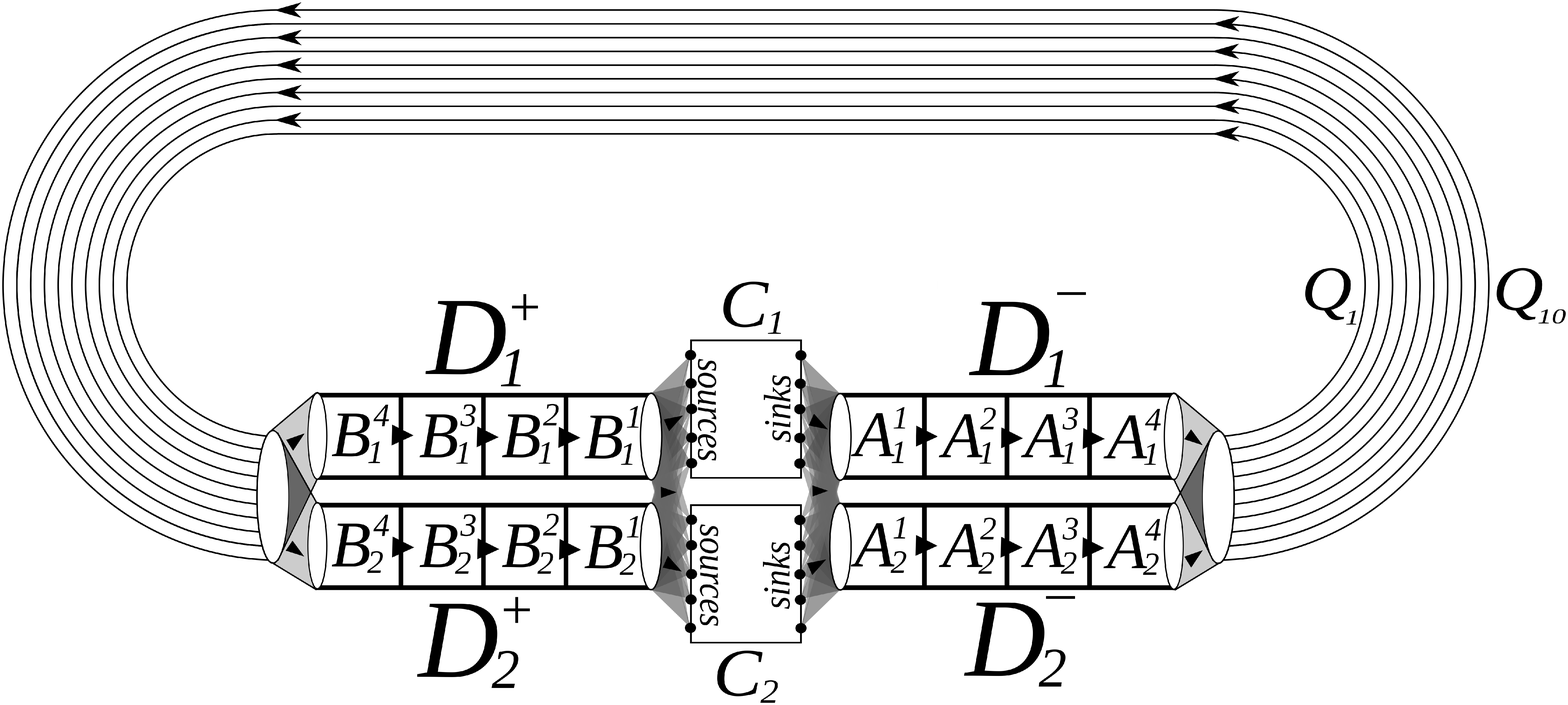}
  \caption{A $2$-linker \label{LinkerFigure}}
\end{figure}

\subsection{Construction of linkers}
The aim of this section is to show that for every $t$, there is a constant $C_0=C_0(t)$ such that every $C_0k$-connected tournament contains $k$ vertex-disjoint $t$-linkers.

The following lemma will be used in our construction of linkers in order to ensure that~(\ref{LinkerEdges}) holds.
\begin{lemma}\label{RamseyLemma}
For all $m, t, {\ell} \in \N$ there exist $R(m,t,{\ell})\in \N$ such that the following holds.
Suppose that $A_1, \dots, A_{R}$ are disjoint sets of vertices of order $2m$ in a tournament~$T$. Then we can choose disjoint sets $I, J \subseteq [R]$, subsets $A'_i\subseteq A_i$ for all $i\in I$, and vertices $v_j\in A_j$ for all $j\in J$ with the following properties.
\begin{itemize}
\item We have $|I|=t$ and $|J|={\ell}$, and $|A_i|=m$ for all $i$. 
\item For all $i,j$, all the edges between $A_i$ and $v_j$ are directed from $A_i$ to $v_j$.
\end{itemize}
\end{lemma}
\begin{proof}
If $R_k(n)$ denotes the $k$-colour Ramsey number,  let $R=R(m,t,{\ell})=R_{2m \binom{2m}{m}} \left(2^{t+{\ell}}\right)$.

For each $i$, let the vertices of $A_i$ be called $a_i^1, \dots, a_i^{2m}$. Notice that for any $i\neq j$ there must be a vertex in either $A_i$ or $A_j$ which has in-degree at least $m$ in the bipartite digraph between $A_i$ and $A_j$. For every $i,j$ choose one such vertex, which we call $v_{i,j}$ and let $N_{i,j}$ be some particular subset of order $m$ of the in-neighbourhood of $v_{i,j}$ in the bipartite digraph between $A_i$ and $A_j$.  

We define a coloured tournament $S$ whose vertex set is $\{1, \dots, R\}$. The edge between $i$ and $j$ in $S$ is directed $ij$  if $v_{i,j}\in A_j$ holds and $ji$ if  $v_{i,j}\in A_i$ holds. In addition we give each edge in $S$ one of $2m \binom{2m}{m}$ colours which are indexed by the set $[2m]\times \binom{[2m]}{m}$. We let the edge $ij$ have colour $(t,X)$ if $v_{i,j}=a_j^t$ and $N_{i,j}=\{a_i^x:x\in X\}$. 

By Ramsey's Theorem, combined with Lemma~\ref{TransitiveSubtournament}, there is a monochromatic transitive subtournament $S'$ of $S$ on $t+\ell$ vertices.  Let $x_1, x_2, \dots, x_{t+\ell}$ be the vertex sequence of $S'$ in the order from head to tail. Let $I=\{x_1, \dots, x_t\}$ and $J=\{x_{t+1}, \dots, x_{t+\ell}\}$. The edges in $S'$ all have the same colour $(t,X)\in [2m]\times \binom{[2m]}{m}$. For each $i\in I$ we let $A'_i=\{a_i^x:x\in X\}$, and for each $j\in J$ we let $v_j=v_j^t$. By the definition of the coloured tournament $S$, this choice of $I$, $J$, $A'_i$s, and $v_j$s satisfy all the conditions of the lemma.
\end{proof}

By reversing arcs in the above lemma, we get the following.
\begin{lemma}\label{RamseyLemmaReverse}
For all $m, t, \ell \in \N$ there exist $R(m,t,l)\in \N$ such that the following holds.
Suppose that $A_1, \dots, A_{R}$ are disjoint sets of vertices of order $2m$ in a tournament~$T$. Then we can choose disjoint sets $I, J \subseteq [R]$, subsets $A'_i\subseteq A_i$ for all $i\in I$, and vertices $v_j\in A_j$ for all $j\in J$ with the following properties.
\begin{itemize}
\item We have $|I|=t$ and $|J|=\ell$, and $|A_i|=m$ for all $i$. 
\item For all $i,j$, all the edges between $A_i$ and $v_j$ are directed from $v_j$ to $A_i$.
\end{itemize}
\end{lemma}

The following technical lemma allows us to find a single $t$-linker in a tournament assuming that we have many disjoint in and out-dominators with paths between them.
\begin{lemma}\label{SingleLinkerConstruction}
For every $t\in \mathbb{N}$, there is a constant $R_0=R_0(t)$ such that the following holds.
Suppose  that $T$ is a tournament, $X$ and $Z$ are subsets of $V(T)$, and $H_1, \dots, H_{R_0}$ are vertex disjoint subdigraphs of $T$ with the following properties.
\begin{enumerate}[(i)]
\item $H_i$ consists of:
\begin{itemize}
\item An $(8,16,16)$-indominator $D^-_i=(A_i^1,A_i^2,A_i^3,A_i^4, E^-_i)$ of $T$ with exceptional set~$X$.
\item An $(8,16,16)$-outdominator $D^+_i=(B_i^1,B_i^2,B_i^3,B_i^4, E^+_i)$ of $T$ with exceptional set~$X$.
\item A path $Q_i$ from the head of $D^-_i$ to the tail of $D^+_i$.
\end{itemize}
In addition, $D^-_i$, $D^+_i$, and the internal vertices of $Q_i$ are all vertex disjoint for each~$i$.
\item All vertices in  $A_i^1$ have large in-degree in $T$. All vertices in  $B_i^1$ have large out-degree in $T$.
\item $V(D^-_i), V(D^+_i) \subseteq  X$ and $V(H_i) \subseteq  Z$ hold for all $i$. 
\item $|Z|\leq |T|/50- 40t$.
\item $2(|X|+40t)\leq \min(\delta^+(T), \delta^-(T))$.
\end{enumerate}
Then  there is a set $S$ with $|S|\leq 40t$ and $S\cap Z=\emptyset$, and a $t$-linker $L$ with exceptional set $X\cup S$ whose vertices are contained in the hypergraphs $H_1, \dots, H_{R_0}$ plus $S$. In addition, for any $Y$ containing $X\cup S$ with $2|Y|\leq \min(\delta^+(T),\delta^-(T))$, $L$ is also a $t$-linker in $T$ with exceptional set $Y$.
\end{lemma}
\begin{proof}
Let $N$ be the constant from Lemma~\ref{ConnectorExistence}. Let $R(m,t,\ell)$ be the function given by Lemma~\ref{RamseyLemma}. We fix the following constants for the proof.
\begin{align*}
R^J_3&=N+40t,  &R^J_4&= N+40t, \\
R^J_2&=5t,  &R^I_5&= R^I_6=t, \\ 
R^I_3&=2R^I_5,   &R^I_4&=2R^I_6, \\
R^I_2 &= R(8, R^I_3, R^J_3), 
&R^I_1 &= R(8, R^I_4, R^J_4), \\  
R^J_1 &= R(8, R^I_2, R^J_2),  
&R_0 &= R(8, R^I_1, R^J_1). 
\end{align*}

Apply Lemma~\ref{RamseyLemma} to the family $\{A_1^4, \dots, A_{R_0}^4\}$ in order to find disjoint sets $I_1$ and $J_1$ such that $|I_1|= R^I_1$, $|J_1|=R^J_1$, and also for each $i\in I_1$  there is an $A'^4_i\subseteq A^4_i$ of order $8$, and for each $j\in J_1$ there is a vertex $v^-_j\in A_j^4$. In addition we have all the edges going from $A'^4_i$ to $v^-_j$ for any $i\in I_1$ and $j\in J_1$. 

Apply Lemma~\ref{RamseyLemmaReverse} to the family $\{B_j^4: j\in J_1\}$ in order to find disjoint sets $I_2$ and $J_2$ contained in $J_1$ such that  $|I_2|= R^I_2$, $|J_2|=R^J_2$, and also for each $i\in I_2$  there is an $B'^4_i\subseteq B^4_i$ of order $8$, and for each $j\in J_2$ there is a vertex $v^+_j\in B_j^4$. In addition we have all the edges going from $v^+_j$ to $B'^4_i$ for any $i\in I_2$ and $j\in J_2$. 

Notice that for each $j\in J_2$ the fact that $A_j^4$ is a transitive tournament implies that there is a path from $v^-_j$ to the start of $Q_j$. Similarly, there is a path in $B_j^4$ from the end of $Q_j$ to $v^+_j$. Joining these two paths to $Q_j$, we obtain a path $Q'_j$ from $v^-_j$ to $v^+_j$ consisting of $Q_j$ and some extra vertices in $A_j^4$ and $B_j^4$.

Apply Lemma~\ref{RamseyLemma}  to the family $\{B_i^1: i\in I_2\}$  in order to find disjoint sets $I_3$ and $J_3$ contained in $I_2$ such that  $|I_3|= R^I_3$, $|J_3|=R^J_3$, and also for each $i\in I_3$  there is an $B'^1_i\subseteq B^1_i$ of order $8$, and for each $j\in J_3$ there is a vertex $u^+_j\in B_j^1$. In addition we have all the edges going from $B'^1_i$ to $u^+_j$ for any $i\in I_3$ and $j\in J_3$. 

Apply Lemma~\ref{RamseyLemmaReverse}  to the family $\{A_i^1: i\in I_1\}$  in order to find disjoint sets $I_4$ and $J_4$ contained in $I_1$ such that  $|I_4|= R^I_4$, $|J_4|=R^J_4$, and also for each $i\in I_4$  there is an $A'^1_i\subseteq A^1_i$ of order $8$, and for each $j\in J_4$ there is a vertex $u^-_j\in A_j^1$. In addition we have all the edges going from $u^-_j$  to $A'^1_i$ for any $i\in I_4$ and $j\in J_4$. 

Recall that for all  $i$, vertices in $A_i^1$ have large in-degree and vertices in $B_i^1$ have large out-degree. In particular this means that $u^-_j$ always has large in-degree and $u^+_j$ always has large out-degree. 
Therefore since $|J_3|, |J_4|\geq N+40t$, we can apply Lemma~\ref{ConnectorExistence} to $T$ with $\{x_1,\dots, x_N\}=\{u^+_j:j\in J_3\}$ and $\{y_1,\dots, y_N\}=\{u^-_j:j\in J_4\}$ in order to find $t$ disjoint connectors $C_1, \dots, C_t$ whose sources are in  $\{u^+_j:j\in J_3\}$ and whose sinks are in $\{u^-_j:j\in J_4\}$ (at each application of Lemma~\ref{ConnectorExistence} we let $Y$ be $Z\setminus \{x_1,\dots, x_N, y_1,\dots, y_N\}$ together with the vertices of the previously constructed connectors. Condition (iv) ensures that $|Y|\leq |T|/50$ as required by Lemma~\ref{ConnectorExistence}.) We let $S=(V(C_1)\cup\dots\cup V(C_t))\setminus Z$.

Notice that for two sets of numbers $A$ and $B$, either half of the numbers in $A$ are larger than half of the numbers of $B$, or half of the numbers of $B$ are larger than half of the numbers of $A$. Applying this with $A=\{|E^-_i|:i\in I_4\}$ and $B=\{|E^+_i|:i\in I_3\}$ gives us two subsets $I_5\subseteq I_3$ and $I_6\subseteq I_4$ with $|I_5|=R^I_5$ and $|I_6|=R^I_6$ such that we either have $|E^-_i|\geq |E^+_j|$ for all $i\in I_6, j\in I_5$, or $|E^-_i|\leq |E^+_j|$ for all $i\in I_6, j\in I_5$.

Now we have everything set up to define our $t$-linker. 
\begin{itemize}
\item The indominators of $L$ are given by $D'^-_i=(A'^1_i,A_i^2,A_i^3,A'^4_i, E^-_i)$ for $i\in I_6$. We reorder these indominators such that $|E^-_1|\geq |E^-_2|\geq \dots \geq |E^-_t|$  holds. 
\item The outdominators of $L$ are given by $D'^+_i=(B'^1_i,B_i^2,B_i^3,B'^4_i, E^+_i)$ for $i\in I_5$. We reorder these outdominators such that  $|E^+_1|\geq |E^+_2|\geq \dots \geq |E^+_t|$ holds.
\item The connectors of $L$ are $C_1, \dots, C_t$.
\item The paths of $L$ are given by $Q'_j$ for $j \in J'_2$.
\end{itemize}
It remains to check that we have constructed everything so that $L$ is  a $t$-linker in $T$ with exceptional set $X\cup S$. 
Notice that   $D'^-_i$ is an $(8,8,16)$-indominator with exceptional set $X$ for each $i$, since $D^-_i$ was an $(8,16,16)$-indominator with exceptional set $X$, and we only removed vertices from the ``$A^1$'' and ``$A^4$'' sets of the indominator. For the same reason $D'^+_i$ is an $(8,8,16)$-outdominator with exceptional set $X$ for each $i$. Lemmas~\ref{DominatorRobustness} and~\ref{OutdominatorRobustness} together with (v) and $|S|\leq 40t$ imply that $D'^-_i$ and $D'^+_i$ are $(8,8,8)$-dominators with exceptional set $X\cup S$. Conditions (L1) -- (L3) are immediate from our construction. Condition (L4) follows from our choice of $I_5$ and $I_6$. Finally all the edges in (L5) are present as a consequence of our applications of Lemmas~\ref{RamseyLemma} and~\ref{RamseyLemmaReverse}.

The fact that ``for any $Y$ containing $X$ with $2|Y|\leq \delta^+(T)$, $L$ is also a linker in $T$ with exceptional set $Y$'' follows immediately from Lemmas~\ref{DominatorRobustness} and~\ref{OutdominatorRobustness}, and the fact that $D'^-_i$ and $D'^+_i$ are $(8,8,16)$-dominators.
\end{proof}

The following lemma allows us to find many linkers in a highly connected tournament.
\begin{lemma}\label{LinkerExistence}
There is a constant $C_0=C_0(t)$  such that every $C_0k$-connected tournament contains $k$ vertex disjoint $t$-linkers with a common exceptional set $X$ of size  $\leq C_0k$.
\end{lemma}
\begin{proof}
We first show that we can find many subdigraphs of $T$ satisfying the conditions of Lemma~\ref{SingleLinkerConstruction}.
Let $R_0=R_0(t)$ be the constant given by Lemma~\ref{SingleLinkerConstruction}.  We let $C_1=50(R_0+40t)$ and $C_0=2^{32}C_1$. Let $T$ be a $C_0k$-connected tournament. Notice that this implies that $|T|, \delta^+(T), \delta^-(T)\geq C_0k$.

\begin{claim}\label{PrelinkerClaim}
The tournament $T$ contains a sets of vertices $X$ and $Z$, and $R_0k$ vertex-disjoint digraphs $H_1, \dots, H_{R_0k}$ satisfying parts (i) -- (v) of Lemma~\ref{SingleLinkerConstruction}.
\end{claim}
\begin{proof}
By applying Lemma~\ref{DominatorExistence} repeatedly, we can choose $C_1k$ vertex-disjoint $(8,16,32)$-indominators $D^-_1, \dots, D^-_{C_1k}$ of $T$, with a common exceptional set $X^-$ of order at most $2^{25}C_1k$. Indeed to do this, we first apply Lemma~\ref{DominatorExistence} to $T$ with $m=8$, $M=16$, $p=64$, $L=2^{25}$, and $Y=\emptyset$ to find an $(8,16,64)$-indominator $D^-_1$ with an exceptional set $X_1$ satisfing $|X_1|\leq 2^{25}$. Then for $i=2,\dots, C_1k$, we apply Lemma~\ref{DominatorExistence} to $T$ with $m=8$, $M=16$, $p=64$, $L=2^{25}$, and $Y=V(D^-_{j})\cup X_{j}$ in order to find a disjoint $(8,16,64)$-indominator $D^-_i$ of $T$ with exceptional set $X_i$ satisfying $|X_i|\leq i2^{25}$ and containing $Y$.  Notice that we always have $|Y|\leq 2^{26}C_1 k \leq |T|/25-2^{2m+2M}$ and so are allowed to apply Lemma~\ref{DominatorExistence} in this way.  Let $X^-=V(D^-_{C_1k})\cup X_{C_1k}$. Notice that $2|X^-|\leq \delta^+(T)$, and so by Lemma~\ref{DominatorRobustness}, for each $i$, $D^-_i$ is an $(8,16,32)$-indominator with exceptional set $X^-$.

By the same argument, using Lemma~\ref{OutdominatorExistence} we can choose  $C_1k$ vertex disjoint $(8,16,16)$-outdominators $D^+_1, \dots, D^+_{C_1k}$ of $T$, with a common exceptional set $X$ of order at most $2^{26}C_1k$ containing $X^-$. By choosing $Y$ to contain $X^-$ at each application of Lemma~\ref{DominatorExistence}, we also ensure that $D^+_i\cap D^-_j=\emptyset$ for all $i$ and $j$. Since $2|X|\leq \delta^+(T)$ holds, Lemma~\ref{DominatorRobustness} again implies that for each $i$, $D^-_i$ is an $(8,16,16)$-indominator with exceptional set $X$.

Recall that Lemma~\ref{DominatorExistence} ensures that all the vertices in $A^1$ of the indominator it produces haveat most large in-degree. Therefore, we have that all the vertices $A_i^1$ and $B_i^1$ have large in-degree and out-degree respectively (as will be required in part (ii) of Lemma~\ref{SingleLinkerConstruction}).

Let $h^-_1, \dots, h^-_{C_1k}$ be the heads of the indominators $D^-_1, \dots, D^-_{C_1k}$. Let $t^+_1,\dots, t^+_{C_1k}$  be the tails  of the outdominators $D^+_1, \dots, D^+_{C_1k}$. Let $T'=\Big(T\setminus \bigcup_{i=1}^{C_1k}\big(V(D^-_i)\cup V(D^+_i)\big)\Big)\cup \bigcup_{i=1}^{C_1k}\{ h^-_i, t^+_i\}$, i.e. $T'$ is the subtournament of $T$ built by removing all the dominators we constructed, and then adding the heads and tails back in.

Since the dominators constructed above each have  $48$ vertices, $T'$ is $(C_0-96C_1)k$-connected. Since $C_0-96C_1\geq C_1$, we can apply  Menger's Theorem to find vertex disjoint paths $Q_1, \dots, Q_{C_1k}$ such that $Q_i$ goes from $h^-_i$ to $t^+_{\sigma(i)}$ for some permutation $\sigma$ of $[C_1k]$.  For each $i$, let $H_i=D^-_i\cup Q_i\cup D^+_{\sigma(i)}$.

Since the graphs $H_i$ are all vertex disjoint, the Pigeonhole Principle implies that there is a subset $I\subseteq [C_1k]$ of order $R_0k=C_1k/50-40tk$ such that $|\bigcup_{i\in I} H_i|\leq |T|/50-40tk$. Let $Z=\bigcup_{i\in I}V(H_i)$.

It is easy to check that the collection of  graphs $\{H_i:i\in I\}$ together with the sets $X$ and $Z$ satisfy all the conditions of Lemma~\ref{SingleLinkerConstruction}.
Indeed (i) and (iii) hold from our construction of the dominators, paths, and sets $X$ and $Z$. Condition (ii) holds since Lemmas~\ref{DominatorExistence} and~\ref{OutdominatorExistence} ensured that all the vertices in $A^1_i$ and $B^1_i$ have large in-degrees and out-degrees respectively. Condition (iv) holds from our choice of $I$. Condition (v) holds since we have $\delta^+(T), \delta^-(T)\geq C_0k\geq 2(|X|+40t)$.
\end{proof}

Now partition $\{H_1, \dots, H_{R_0+k}\}$ into $k$ collections $\mathcal H_j=\{H_{(j-1)R_0+1}, \dots, H_{jR_0}\}$ for $j=1, \dots, k$. Let $X_1=X$ and $Z_1=Z$. Apply Lemma~\ref{SingleLinkerConstruction} to $\mathcal H_1$ with the sets $X_1$ and $Z_1$ to find a $t$-linker consisting of vertices in $\mathcal H_1$, plus a set of vertices $S_1$ of order at most $40t$. Let $X_2=X_1\cup S_1$ and $Z_2=Z_1\cup S_1$. Then for each $i=2, \dots, k$, apply Lemma~\ref{SingleLinkerConstruction} to $\mathcal H_i$ with the sets $X_i$ and $Z_i$ to find a $t$-linker $L_i$ consisting of vertices in $\mathcal H_i$, plus a set of vertices $S_i$ of order at most $40t$ (at each step letting $X_i=X_{i-1}\cup S_{i-1}$ and $Z_i=Z_{i-1}\cup S_{i-1}$). This gives us a collection of $k$ disjoint linkers $L_1, \dots, L_{k}$ with exceptional sets $X_1,\dots, X_{k} $ respectively.  Since $|X|\leq 2^{26}C_1k$, $|S_i|\leq 40t$ and $X_{k}=X\cup S_1\cup \dots\cup S_{k}$ we have $|X_{k}|\leq 2^{27}C_1k\leq C_0k$. The last part of Lemma~\ref{SingleLinkerConstruction} ensures that $L_1, \dots, L_{k}$  are all $t$-linkers in $T$ with the common exceptional set $X_{k}$ as required. 
\end{proof}

\subsection{Properties of linkers}
In this section, we prove that families of linkers are linking families.
First we will need to show that linkers have Hamiltonian paths between pairs of essential vertices.
\begin{lemma}\label{1LinkerHamiltonicity}
Let $L$ be a $1$-linker in a tournament $T$. Let $x$ and $y$ be two distinct vertices in $L$ such that $x$  is in  the indominator of $L$  and $y$ is either in the outdominator of $L$ or a sink of the connector
 of $L$. Then $L$ contains a Hamiltonian path from $x$ to $y$.
\end{lemma}
\begin{proof}
Let $D^-=(A^1,A^2,A^3,A^4,E^-)$, $D^+=(B^1,B^2,B^3,B^4, E^+)$, and $C$ be the indominator, outdominator, and connector of $L$ respectively, and $Q_1, Q_2, Q_3, Q_4, Q_5$ be the five paths of $L$.

First we'll consider the case when $y$ is in the outdominator of $L$.  
Let $P_x$ be a shortest path from $x$ to $A^4$. Let $P_y$ be a shortest path from $B^4$ to $y$. Let $P^-_1$, $P^-_2$, $P^-_3$, $P^-_4$  be four paths, each from $A^1$ to $A^4$ such that $P_x$, $P^-_1$, $P^-_2$, $P^-_3$, $P^-_3$ together partition $V(D^-)$ (we can choose such disjoint paths using (D\ref{Indom:Transitive}) combined with the fact that $|A^i|= 8$ for all~$i$).
 Similarly, let $P^+_1$, $P^+_2$, $P^+_3$, $P^+_4$  be four paths, each from $B^4$ to $B^1$ such that $P_y$, $P^+_1$, $P^+_2$, $P^+_3$, $P^+_4$ together partition $V(D^+)$.  From the definition of connector, we can partition $V(C)$ into four paths $R_1, \dots, R_4$, each going from a source of $C$ to a sink. 
 Now we have a Hamiltonian path from  $x$ to $y$ formed by joining 
 $P_x$ to  $Q_1$ to  $P^+_1$ to $R_1$ to $P^-_1$ to $Q_2$  to $P^+_2$ to $R_2$ to $P^-_2$ to $Q_3$ to $P^+_3$ to $R_3$ to $P^-_3$ to $Q_4$ to $P^+_4$ to $R_4$ to $P^-_4$ to $Q_5$ to $P_y$. Part (L5) of Definition~\ref{LinkerDefinition} ensures that all the edges between the endpoints of these paths are oriented the correct way.

Now consider the case when $y$ is a sink of $C$. As in the previous case, let $P_x$ be a shortest path from $x$ to $A^4$, let $P^-_1$, $P^-_2$, $P^-_3$, $P^-_4$  be four paths from $A^1$ to $A^4$  partitioning $V(D^-)$,  let $P^+_1$, $P^+_2$, $P^+_3$, $P^+_4$, $P^+_4$  be five paths, from $B^4$ to $B^1$ partitioning $V(D^+)$. From the definition of connector, we can partition $V(C)$ into five paths $R_1, \dots, R_5$, each going from a source of $C$ to a sink. Since $y$ is a sink, one of these paths ends in $y$ . Without loss of generality let this be $R_5$.  Now we have a Hamiltonian path from $x$ to $y$  formed by joining $P_x$ to $Q_1$ to $P^+_1$ to $R_1$ to $P^-_1$ to $Q_2$ to $P^+_2$ to $R_2$ to $P^-_2$ to $Q_3$ to $P^+_3$ to $R_3$ to $P^-_3$ to $Q_4$ to $P^+_4$ to $R_4$ to $P^-_4$ to $Q_5$ to $P^+_5$ to $R_5$.
\end{proof}

\begin{lemma}\label{LinkerHamiltonicity}
Let $L$ be a $t$-linker for any $t\geq 1$. Let $x$ be a vertex in one of the indominators of $L$ and $y$ a vertex in one of the outdominators of $L$. Then $L$ contains a Hamiltonian path from $x$ to $y$.
\end{lemma}
\begin{proof}
If $t=1$, then the lemma follows from Lemma~\ref{1LinkerHamiltonicity}, so suppose $t\geq 2$.
We can partition $L$ into $t$ $1$-linkers $L_1, \dots, L_t$ such that $L_1$  contains $x$ and $L_t$ contains~$y$.

By Lemma~\ref{1LinkerHamiltonicity}, for we can find a Hamiltonian path ${P}_1$ in $L_1$ from $x$ to  a sink of the connector of $L_1$. Similarly for $i=2, \dots, t-1$, we can find a Hamiltonian path ${P}_i$ from a vertex in the $A^1$-set of $L_i$ to a sink of the connector of $L_i$. Finally, we can find a Hamiltonian path ${P}_t$ in $L_t$ from a vertex in the $A^1$-set of $L_t$ to $y$. Joining these together using the fact that there is  an edge from any of the sinks of the connectors in a $t$-linker and the $A^1$-sets, gives the required Hamiltonian path in $L$.
\end{proof}

The following lemma is the main property that linkers have. It says that under certain conditions on a tournament $T$, a linker is a linking family of size 1 in $T$.
\begin{lemma}\label{LinkerMainLemma}
Let $t$ and $K$ be integers satisfying $K/5\geq t\geq 12$.
Let $T$ be a tournament with minimum in and out-degrees at least $80K$.
Suppose that  we have a $t$-linker $L$ in $T$ with exceptional set $X$ such that $|X|\leq K$.

For $r\leq K$, suppose we have two vertices $x$ and $y \in V(T)\setminus V(L)$ and vertex disjoint paths ${P}_1, \dots, {P}_r$  in $V(T)\setminus (V(L)\cup \{x,y\})$.  Then there are vertex disjoint paths ${P}, {P}'_1, \dots, {P}'_r$ such that
\begin{enumerate}[(i)]
\item ${P}$ is from $x$ to $y$.
\item ${P}'_j$ has the same endpoints as ${P}_j$ for every $j$.
\item $V({P})\cup V({P}'_1)\cup \dots\cup V({P}'_r)$ consists of $V(L_i)\cup V({P}_1)\cup \dots\cup V({P}_r)\cup\{x,y\}$, plus at most $6$ other vertices.
\end{enumerate}
\end{lemma}
\begin{proof}
We'll actually prove a slightly stronger statement about $t$-linkers for all $t\geq 1$. Suppose we have $t\geq 1$, and $K\geq 5t$, $T$, $L$ and $X$ as in the statement of the lemma.  Let the dominators, connectors, and paths of $L$ be labelled as in the Definition~\ref{LinkerDefinition}.

Notice that without loss of generality, we can assume that $|E_t^-|\geq |E_1^+|$ occurs in (L\ref{UncoveredBiggerLinker}) for the linker $L$. Indeed otherwise, we could reverse all arcs in the tournament and exchange the roles of $x$ and $y$ in order to reduce to the case when $|E_t^-|\geq |E_1^+|$ holds.

Let $x$ and $y$ be two vertices in $V(T)\setminus V(L)$. We will prove the lemma in several steps depending on where $x$ and $y$ lie.
\begin{claim}
Let  ${P}_1, \dots, {P}_r$  be vertex disjoint paths in $V(T)\setminus (V(L)\cup \{x,y\})$. 
 Suppose that any of the following hold.
\begin{enumerate} [(a)]
\item $t\geq 1$, $r\leq K+9$, $m=0$, and $x\not \in E^-_i\cup X$ and $y\not \in E^+_j\cup X$  for some $i,j\leq t$.
\item $t\geq 2$, $r\leq K+4$, $m=1$, and $x \not \in X$ and $y\not \in E^+_j\cup X$ for some $j\leq t$.
\item $t\geq 4$, $r\leq K+2$, $m=2$, and $x \not \in  X$ and $y\not \in X$.
\item $t\geq 12$, $r\leq K$, and $m=6$.
\end{enumerate}

Then there are vertex disjoint paths ${P}, {P}'_1, \dots, {P}'_r$ such that
\begin{enumerate}[(i)]
\item ${P}$ is from $x$ to $y$.
\item ${P}'_j$ has the same endpoints as ${P}_j$ for every $j$.
\item $V({P})\cup V({P}'_1)\cup \dots\cup V({P}'_r)$ consists of $V(L_i)\cup V({P}_1)\cup \dots\cup V({P}_r)\cup\{x,y\}$, plus at most $m$ other vertices.
\end{enumerate}
\end{claim}
\begin{proof}
Let $Q_1, \dots, Q_{5t}$ be the paths of $L$
Let $U$ be the set of endpoints of the paths ${P}_1, \dots, {P}_r$, and $W$ the set of endpoints of the paths $Q_1, \dots, Q_{5t}$. Notice that we have $|U|, |W|\leq 4K$.

\begin{enumerate}[(a)]
\item Since  $x\not \in E^-_i\cup X$, there is some $x_1\in D^-_i$ such that $xx_1$ is an edge. Similarly, since $y\not \in E^+_j \cup X$, there is some $y_1\in D^+_j$ such that $y_1 y$ is an edge. Applying  Lemma~\ref{LinkerHamiltonicity} to $L$ gives us a Hamiltonian path $R$ in $L$ from $x_1$ to $y_1$. Letting ${P}$ be the path formed by joining $x$ to $R$ to $y$ and ${P}'_i={P}_i$ for every $i$ proves the claim.

\item 
If $x \not\in E^-_1$ then we are done by  part (a). Therefore suppose that we have $x\in E^-_1$.
Choose $\ell$ to be any integer between $1$ and $t$ which is not $j$.

Since $x\in E^-_1$ we have that $|N^+(x)|\geq 8|E^-_1|$.   We also have $d^+(x)\geq 80K$. Averaging these and using $|X|\leq K$ and $|U|, |W|\leq 4K$ we obtain $|N^+(x)|\geq  4(|E^-_1|+|X|+|U|+|W|)$. Therefore there are at least $3(|E^-_1|+|X|+|U|+|W|)$ vertices in $|N^+(x)|$ outside of $E^-_1\cup X\cup U\cup W$. If one of these vertices, $x'$, is not on any of the paths ${P}_1, \dots, {P}_r, Q_1, \dots, Q_{5t}$ then we can let ${P}_{r+1}=\{x\}$, and apply part (a) to get a path $Q$ from $x'$ to $y$ and then join $x$ to this path to prove the claim.  

Therefore, we can suppose that all the vertices in $N^+(x) \setminus (E^-_1\cup X\cup U\cup W)$ are on the paths ${P}_1, \dots, {P}_r, Q_1, \dots, Q_{5t}$.  Since $|N^+(x)\setminus (E^-_1\cup X\cup U\cup W)|\geq  3(|E^-_1|+|X|+|U|+|W|)$ holds and $|E^-_1|\geq |E^-_{\ell}|, |E^+_{\ell}|$, we can choose a vertex $x_1$ in $N^+(x)\setminus (E^-_1\cup X\cup U\cup W)$ such that $x_1$ is on a path $Q'\in \{Q_1, \dots, Q_{5t}, {P}_1,\dots, {P}_{r}\}$, the predecessor of $x_1$ on $Q'$ is not in $E^-_{\ell}\cup X$, and the successor of $x_1$ on $Q'$ is not in $E^+_{\ell}\cup X$.
We'll suppose for now that $Q'$ is one of the paths $Q_1, \dots, Q_{5t}$. Without loss of generality $Q'=Q_{5t}$.
Let $x_2$ be the predecessor of $x_1$ on $Q'$ and $y_2$ the successor of $x_1$ on $Q'$. Let $Q'_x$ be the initial segment of $Q'$ ending at the predecessor of $x_2$ and $Q'_y$ the final segment of $Q'$ starting at the successor of~$y_2$.

Let $L'$ be a  $1$-linker contained in $L$ consisting of $D^-_{\ell}$, $D^+_{\ell}$,  the connector $C_{\ell}$, and  the paths $Q_1, \dots, Q_5$.  Now let $T'$ be the subtournament of $T$ formed by removing the essential vertices of $L$, and adding the essential vertices of $L'$ back in. It is easy to check that $L'$ is still a $1$-linker in $T'$ (using the fact that all the vertices we removed from $T$ were in the exceptional set $X$). Apply part (a) to $T'$ with the vertices $x_2$, $y_2$,  $1$-linker $L'$, and paths $\{Q'_x, Q'_y\}\cup\{{P}_1, \dots, {P}_r, Q_6, \dots, Q_{5t-1}\}$ as well as three one-vertex paths $\{x\}$, $\{y\}$, and $\{x_1\}$. This gives us disjoint paths $Q''_x, Q''_y, {P}'_1, \dots, {P}'_r, Q'_6, \dots, Q'_{5t-1}$ with the same endpoints as the previous paths, and a new path $R$ starting at $x_2$ and ending at~$y_2$. In addition all these paths avoid $x$, $y$, and $x_1$, and the union of their vertices is  $V(L')\cup V(P_1)\cup \dots \cup V(P_r)\cup V(Q_1)\cup \dots \cup V(Q_{5t-1})\cup V(Q'_x)\cup V(Q'_y)\cup \{x_2,y_2\}$. Let $Q_{5t}'$ be the path formed by joining $Q''_x$ to $R$ to $Q''_y$.  

Let $L''$ be the $(t-1)$-linker  formed from $L$ by removing $L'$ and replacing $Q_i$ by $Q'_i$ for each $i$. Now we can apply part (a) in $T$ with the linker $L''$, vertices $x_1$ and $y$,  and paths ${P}'_1, \dots, {P}'_r, \{x\}$. This gives us paths ${P}''_1, \dots, {P}''_r$ as well as a path $P$ from $x_1$ to~$y$. Joining $x$ to $R'$ gives the required collection of paths.

The case when $Q'$ was one of the paths  ${P}_1,\dots, {P}_{r}$ is proved identically.

\item
If $y \not\in E^+_1$ then we are done by  part (b). Therefore suppose that we have $y\in E^+_1$.

Since $y\in E^+_1$ we have that $|N^-(y)|\geq 8|E^+_1|$. As before, there are at least $3(|E^+_1|+|X|+|U|+|W|)$ vertices in $|N^-(y)|$ outside of $E^+_1\cup X\cup U\cup W$. If one of these vertices,~$y'$, is not on any of the paths ${P}_1, \dots, {P}_r, Q_1, \dots, Q_{5t}$ then we can let $P_{r+1}=\{y\}$, and apply part (b) to get a path $Q$ from $x$ to $y'$ and then join this path to $y$ to prove the claim.  

Therefore, we can suppose that all the vertices in $N^-(x) \setminus (E^+_1\cup X\cup U\cup W)$ are on  the paths ${P}_1, \dots, {P}_r, Q_1, \dots, Q_{5t}$.  Since $|N^-(y)\setminus (E^+_1\cup X\cup U\cup W)|\geq  3(|E^+_1|+|X|+|U|+|W|)$ and $|E_1^+|\geq |E_2^+|$ hold, we can choose a vertex $y_1$ in $N^-(x)\setminus (E^+_1\cup X\cup U\cup W)$ such that $y_1$ is on a path $Q'\in \{Q_1, \dots, Q_r, {P}_1,\dots, {P}_{5t}\}$ for some $i$, and the neighbours of $y_1$ on this path are  in neither $|E^+_2|$ nor $X$. Let $x_2$ be the predecessor of $y_1$ on $Q'$ and $y_2$ the successor of $y_1$ on $Q'$. 

The rest of the proof is nearly identical to the proof of part (b), so we only sketch it. We choose a $2$-linker $L'$ contained in $L$ such that $D^+_2$ is one of the outdominators of $L'$ and $L'$ doesn't contain the path $Q'$. We remove the essential vertices of the linker $L$ from $T$ and add $L'$ back in to obtain a tournament $T'$. Apply part (b) to $T'$ with the linker $L'$ in order to join $x_2$ to $y_2$ by a path. Then let $L''$ be the $(t-2)$-linker in $T$ formed by removing $L'$ from $L$. Applying part (b) to $T$ with the linker $L''$ allows us to join $x$ to $y_1$ (and then to $y$) as required.

\item 
Notice that since $T$ has minimum out-degree $\geq 80K$, $x$ has at least $7|X\cup U\cup W|$ out-neighbours outside of $X\cup U\cup W$. 
Similarly,  since $T$ has minimum in-degree $\geq 80K$, $y$ has at least $7|X\cup U\cup W|$ in-neighbours outside of $X\cup U\cup W$.  
Suppose for now that all such neighbours of $x$ and $y$ lie on the paths ${P}_1, \dots, {P}_r, Q_1, \dots, Q_{5t}$. 
Then we can choose an out-neighbour $x_1$  of $x$, and a distinct in-neighbour $y_1$  of $y$, such that $x_1$ and $y_1$ are outside of $X\cup U\cup W$. 
In addition, $x_1$ will lie on some path with predecessor $x_2$ and successor $y_2$, and $y_1$ will lie on some path with predecessor $x_3$ and successor $y_3$, such that $x_2,y_2,x_3, y_3\not\in X$. 

Similarly to how we did in (b) and (c), we can partition the linker $L$ into three sublinkers, and then apply part (c) three times in order to join $x_2$ to $y_2$ then $x_3$ to $y_3$, and finally $x_1$ to $y_1$.

The cases when $x$ and/or $y$ have neighbours outside of  $X\cup U\cup W$ and the paths ${P}_1, \dots, {P}_r, Q_1, \dots, Q_{5t}$ is very similar. The only difference is that we might not have the vertices $x_2, y_2$ or $x_3, y_3$ to join, and so we would be able to find the required $x$ -- $y$ path using either one or two applications of part (c).
\end{enumerate}
\end{proof}
The lemma follows since it is exactly part (d) of the claim.
\end{proof}

So far we have only considered linkers in tournaments.
In the Theorem~\ref{LinkageStructures}, we will actually need linkers in digraphs. We say that $L$ is a $t$-linker in a digraph $D$  with exceptional set $X$ if there is some tournament on the vertices $V(D)$ containing $D$ in which $L$ is a $t$-linker in $T$ with exceptional set $X$ (and we also require that $D$ contains all the edges of~$L$). 
We'll need the following version of Lemma~\ref{LinkerMainLemma} for digraphs.
\begin{lemma}\label{LinkerMainLemmaDigraph}
Let $t$ and $K$ be integers satisfying $K/5\geq t\geq 12$.
Let $D$ be a digraph with minimum degree at least $|D|-K$ and  minimum in and out-degrees at least $81K$.
Suppose that  we have a $t$-linker $L$ in $D$ with exceptional set $X$ such that $|X|\leq K$.

For $r\leq K$, suppose we have two vertices $x$ and $y \in V(D)\setminus V(L)$ and vertex disjoint paths ${P}_1, \dots, {P}_r$  in $V(D)\setminus (V(L)\cup \{x,y\})$.  Then there are vertex disjoint paths ${P}, {P}'_1, \dots, {P}'_r$ such that
\begin{enumerate}[(i)]
\item ${P}$ is from $x$ to $y$.
\item ${P}'_j$ has the same endpoints as ${P}_j$ for every $j$.
\item $V({P})\cup V({P}'_1)\cup \dots\cup V({P}'_r)$ consists of $V(L_i)\cup V({P}_1)\cup \dots\cup V({P}_r)\cup\{x,y\}$, plus at most $6$ other vertices.
\end{enumerate}
\end{lemma}

The above lemma has an identical proof to Lemma~\ref{LinkerMainLemma}. The only difference is that all vertices have slightly smaller degree, but given that all vertices have minimum in and out-degree $81K$, this is not significant in any of the inequalities in the proof of Lemma~\ref{LinkerMainLemma}. Given the similarity between Lemmas~\ref{LinkerMainLemma} and \ref{LinkerMainLemmaDigraph}, we omit the proof of Lemma~\ref{LinkerMainLemmaDigraph}.

Now we use Lemma~\ref{LinkerMainLemmaDigraph} to prove that a family of linkers is a linking family.
\begin{lemma}\label{LinkerLinkingFamily}
Let $t$ and $K$ be integers satisfying $K/5\geq t\geq 12$.
Let $D$ be a digraph with minimum degree at least $|D|-K/2$ and  minimum in and out-degrees at least $82K$.

Suppose that for $k$ satisfying $104tk\leq K/2$,  we have a family of vertex-disjoint $t$-linkers $L_1, \dots, L_k$ in $D$ with common exceptional set $X$ such that $|X|\leq K$.
Then $\{L_1, \dots, L_k\}$ is a linking family in $D$.
\end{lemma}
\begin{proof}
The proof is by induction on $k$. Suppose that the statement is false. Let $k_0$ be the minimal value of $k$ for which it is false.

Let $L_1, \dots, L_{k_0}$ be a family of $k_0$ vertex-disjoint $t$-linkers with common exceptional set $X$ as in the lemma. Let $\mathcal{Q}_1, \dots, \mathcal Q_{k_0}$ be the families of paths of these linkers. Let $x,y$ be two vertices, $m\leq 100k_0$, and  $P_1, \dots, P_m$ paths as in the definition of ``linking family''.

Let $D'$ be $D$ with the essential vertices of $L_1, \dots, L_{k_0-1}$ removed. 
Notice that $D'$ has minimum degree at least $|D|-K/2-104t k_0\geq |D|-K$ and minimum in and out-degrees at least $82K-104t k_0\geq 81K$. Also notice that the total number of paths in $\{P_1, \dots, P_m\}\cup {Q}_i\cup \dots \cup Q_{k_{0}-1}$ is at most $100k_0+5t(k_0-1)\leq K$.
Therefore we can apply Lemma~\ref{LinkerMainLemmaDigraph} in $D'$ with the linker $L_{k_0}$, vertices $x$ and $y$, and paths $P_1, \dots, P_m$, plus all the paths in $\mathcal{Q}_1, \dots, \mathcal Q_{k_0-1}$. This gives us an $x$ -- $y$ path $P$, and new paths $P'_1, \dots, P'_m$, and families of paths $\mathcal{Q}'_1, \dots, \mathcal Q'_{k_0}$. Since for each $i$, the paths in $\mathcal{Q}'_i$ have the same endpoints as those in $\mathcal{Q}_i$  we can define a new $t$-linker $L'_i$ formed by replacing the paths in $L_i$ with those in $\mathcal{Q}'_i$.

Now we claim that the paths $P$, $P'_1, \dots, P'_m$, and digraphs $L'_1, \dots, L'_{k_0}$ satisfy (i) -- (iv) in the definition of ``linking family''. Conditions (i) -- (iii) are immediate from our application of Lemma~\ref{LinkerMainLemmaDigraph}. If $k_0>1$, then (iv) holds  by minimality of $k_0$, and if $k_0=1$ then (iv)  holds vacuously. This shows that  $\{L_1, \dots, L_{k_0}\}$ is a linking family in $D$. Since we made no assumptions on $\{L_1, \dots, L_{k_0}\}$ and $D$ other than those in the lemma, this contradicts our assumption that the lemma was false for $k=k_0$.
\end{proof}

\subsection{Proof of Theorem~\ref{LinkageStructures}}
Putting together Lemmas~\ref{LinkerExistence} and~\ref{LinkerMainLemma} it is easy to prove Theorem~\ref{LinkageStructures}.
\begin{proof}[Proof of Theorem~\ref{LinkageStructures}]
 Let $C_0=C_0(12)$ be the constant from Lemma~\ref{LinkerExistence}. Let $\Delta_1$ be the maximum degree of a $12$-linker. Set $C_1=8300\Delta_1C_0$.

 Let $T$ be any $C_1k$ connected tournament. By Lemma~\ref{LinkerExistence}, we can find $k$ vertex-disjoint $12$-linkers $L_1, \dots, L_k$ in $T$ with a common exceptional
 set $X$ satisfying $|X|\leq C_0k$.

Let $K=100\Delta_1C_0k$.
Let $D$ be a subdigraph of $T$ satisfying $\delta(D)\geq |T|-50\Delta_1C_0 k=|T|-K/2$. 
Notice that since $T$ is $C_1$-connected, it must satisfy $\delta^-(T), \delta^+(T)\geq C_1k$ and so  $\delta^-(D), \delta^+(D)\geq (C_1-50\Delta_1C_0)k\geq 82K$. Now we can apply Lemma~\ref{LinkerLinkingFamily} in order to conclude that for any subfamily $\hyper{L}\subseteq \{L_1, \dots, L_k\}$ is a linking family in $D\cup \hyper L$.
\end{proof}


\section{Concluding remarks}
We close with some remarks and open problems.
\begin{itemize}
\item For clarity of presentation, we made no attempt to optimize the constant $C$ in Theorem~\ref{EdgeDisjointHamiltonianCycles}. In future work it might be interesting to investigate how small this constant can be made, or to see whether exact bounds on the connectivity can be obtained for small $k$. For $k=2$, Thomassen conjectured that every strongly $3$-connected tournament contains $2$ edge-disjoint Hamiltonian cycles~\cite{ThomassenHamiltonian}.

\item A tournament is $k$-linked if  for any two disjoint sets of vertices $\{x_1, \dots, x_k\}$ and $\{y_1, \dots, y_k\}$ there are vertex disjoint paths $P_1, \dots, P_k$ such that $P_i$ goes from $x_i$ to $y_{i}$. Recall that an important step of the proof of Theorem~\ref{HamiltonianTheoremKLOP} in~\cite{KLOP} is to first show that a highly connected tournament is highly linked, and then to proceed to construct linkage structures under the knowledge that the tournament is linked. In our proof of Theorem~\ref{EdgeDisjointHamiltonianCycles} we used only connectedness and not linkedness. 

Interestingly, Theorem~\ref{LinkageStructures} can be used to show every highly connected tournament is highly linked---specifically we can show that there is a constant $C$ such that every $Ck$-connected tournament is $k$-linked. Indeed letting $C=3C_1$, we have that every $Ck$-connected tournament contains a family of $3k$ digraphs $L_1, \dots, L_{3k}$ any subfamily of which form a linking family in $T$. Given sets of vertices $\{x_1, \dots, x_k\}$ and $\{y_1, \dots, y_k\}$ as in the definition of $k$-linkedness,  at least $k$ of the graphs $L_1, \dots, L_{3k}$ must be  disjoint from $\{x_1, \dots, x_k, y_1, \dots, y_k\}$. Without loss of generality, these are the graphs $L_1, \dots, L_{k}$. Now invoking the property of the linking family $L_1, \dots, L_{k}$ with the vertices $x_1$ and $y_1$, and paths $P_1=\{x_2\}$, $P_3=\{y_2\}$, $P_4=\{x_3\}$, $P_5=\{y_3\}$, $\dots,$ $P_{2k-3}=\{x_k\}$, $P_{2k-2}=\{y_k\}$, we obtain a path $P$ from $x_1$ to $y_1$ which is disjoint from $\{x_1, \dots, x_k, y_1, \dots, y_k\}$ as well as a new linking family $\{L'_1, \dots, L'_{k}\}$ of size $k-1$. Next, invoking the property of the linking family $L'_1, \dots, L'_{k}$ with the vertices $x_2$ and $y_2$, and paths $P, P_1=\{x_3\}$, $P_3=\{y_3\}$, $P_4=\{x_4\}$, $P_5=\{y_4\}$, $\dots,$ $P_{2k-5}=\{x_k\}$, $P_{2k-4}=\{y_k\}$, we obtain an $x_2$ to $y_2$ path $Q$, a disjoint $x_1$ to $y_1$ path $P'$ and a new linking family of size $k-2$. Continuing in this fashion produces disjoint $x_i$ to $y_i$ paths for all $i$.

The above argument shows that there is a constant $C$ such that every $Ck$-connected tournament is $k$-linked. This was a conjecture of K\"uhn, Lapinskas, Osthus, and Townend from~\cite{KLOP}.
In \cite{PokrovskiyLinking}, the author gave a  proof of this conjecture with the constant $C=452$.  The proof of this result also uses linkage structures, but is much shorter than the one that is obtained in this paper from Theorem~\ref{LinkageStructures}. In addition the constant ``$452$'' is better than the one that would be obtained from Theorem~\ref{LinkageStructures}.

\item
There are a several open problems in this area.
One is the following conjecture of K\"uhn, Osthus, and Townend.
\begin{conjecture}[K\"uhn, Osthus, and Townend, \cite{KOT}]
There is a constant $C$ such that the vertices of every strongly $Ctk$-connected tournament can be partitioned into $t$  strongly  $k$-connected subtournaments.
\end{conjecture}
The existence of a function $f(t,k)$ for which every strongly $f(t,k)$-connected tournament can be partitioned into $t$  strongly  $k$-connected subtournaments was a conjecture of Thomassen. This conjecture was solved by K\"uhn, Osthus, and Townend using a version of Theorem~\ref{LinkageStructuresKOT}. The only $k$ for which a linear bound is known is $k=1$, where $f(t,1)=t$ was proved by Chen, Gould, and Li \cite{CGL}.

Another is a conjecture of Song \cite{Song}, which says that for any natural numbers $n_1, \dots, n_k$ satisfying $\sum_{i=1}^k n_i=n$, every sufficiently large $k$-connected tournament $T$ on $n$ vertices can be partitioned into cycles $C_1,\dots, C_k$ such that $|C_i|=n_i$. K\"uhn, Osthus, and Townend showed that this is true with the condition that ``$T$ is $k$-connected'' is replaced by ``$f(k)$-connected'' for a suitable function $f(k)$. As an intermediate step to Song's conjecture it would be interesting to show that $f(k)$ can be linear.
\begin{problem}
Show that there is a constant $M$, such that for any natural numbers $n_1, \dots, n_k$ satisfying $\sum_{i=1}^k n_i=n$, the vertices of every strongly $Mk$-connected tournament $T$ on $n$ vertices can be partitioned into cycles $C_1,\dots, C_k$ such that $|C_i|=n_i$.
\end{problem}

Finally, as a tool for studying the above conjectures it would be interesting to know how small the bound on the connectivity in Theorem~\ref{LinkageStructuresKOT} can be.
\end{itemize}

\bibliography{Linking}
\bibliographystyle{abbrv}
\end{document}